\documentclass[12pt,twoside,reqno,english]{amsart}
\usepackage[T1]{fontenc}
\usepackage[latin1]{inputenc}
\usepackage{amsthm}
\usepackage{graphicx}
\usepackage{bbm}
\usepackage{setspace}
\usepackage{color}
\usepackage{enumitem}
\usepackage{cancel}
\usepackage[left=2.3cm,top=3cm,right=2.3cm]{geometry}
\usepackage{microtype, float}

\makeatletter

\usepackage{babel}

\let\oldtocsection=\tocsection
\let\oldtocsubsection=\tocsubsection
\let\oldtocsubsubsection=\tocsubsubsection
\renewcommand{\tocsection}[2]{\hspace{0em}\textbf{\oldtocsection{#1}{#2}}}
\renewcommand{\tocsubsection}[2]{\hspace{1.8em}\oldtocsubsection{#1}{#2}}
\renewcommand{\tocsubsubsection}[2]{\hspace{3em}\oldtocsubsubsection{#1}{#2}}

\usepackage[colorlinks = true, citecolor = black, linkcolor = black, urlcolor = black, pdfstartview=FitH]{hyperref}

\numberwithin{equation}{section}

\theoremstyle{plain}
\newtheorem{theorem}{Theorem}[section]
\newtheorem{corollary}[theorem]{Corollary}

\newtheorem{proposition}[theorem]{Proposition}
\newtheorem{lemma}[theorem]{Lemma}
\newtheorem{lem}[theorem]{Lemma}

\theoremstyle{remark}

\newtheorem{rem}[theorem]{Remark}

\theoremstyle{definition}
\newtheorem{definition}[theorem]{Definition}

\newtheorem*{notation*}{Notation}

\usepackage{amssymb,amsmath}

\DeclareMathOperator{\diam}{diam}

\renewcommand{\phi}{\varphi}

\renewcommand{\tilde}{\widetilde}
\renewcommand{\emptyset}{\varnothing}
\newcommand{\eps}{\varepsilon}

\def\diam {\mathop {\hbox{\rm diam}}}

\def\R{\mathbb{R}}
\def\\Sigma{\mathbb{H}}

\def\N{\mathbb{N}}

\def\d{\delta}

\SetSymbolFont{operators}{bold}{OT1}{cmr}{bx}{n}
\SetSymbolFont{letters}{bold}{OML}{cmm}{b}{it}
\SetSymbolFont{symbols}{bold}{OMS}{cmsy}{b}{n}

\newcommand{\cM}{\mathcal{M}}
\newcommand{\cH}{\mathcal{H}}

\newcommand{\var}{\mathrm{var}}
\newcommand{\Hd}{\dim_\mathrm{H}}
\renewcommand{\i}{\mathbf{i}}

\renewcommand{\v}{\mathbf{v}}
\renewcommand{\j}{\mathbf{j}}
\renewcommand{\k}{\mathbf{k}}
\renewcommand{\d}{\,{\, d}}

\DeclareMathOperator{\llocd}{\underline{dim}_{loc}}

\DeclareMathOperator*{\essinf}{ess\,inf}
\usepackage[usenames,dvipsnames]{xcolor}
\newcommand{\red}[1]{ {\color{Red}#1} }

\numberwithin{equation}{section} 

\begin{document}

\title{Pointwise perturbations of countable Markov maps}

\author{Thomas Jordan}
\address{School of Mathematics, University of Bristol, University Walk, Clifton, Bristol, BS8 1TW, England}
\email{thomas.jordan@bristol.ac.uk}

\author{Sara Munday}
\address{Dipartimento di Matematica, Universit\`{a} di Bologna, Piazza di Porta S.Donato 5, 40126 Bologna, Italy}
\email{magicdairyfairy@gmail.com}

\author{Tuomas Sahlsten}
\address{Department of Mathematics, University of Bristol, University Walk, Clifton, Bristol, BS8 1TW, England}
\email{tuomas.sahlsten@bristol.ac.uk}

\dedicatory{Dedicated to the memory of Bernd O. Stratmann}



\thanks{TS is supported by the European Union (ERC grant $\sharp$306494 and MSCA-IF grant $\sharp$655310)}

\begin{abstract}
We study the pointwise perturbations of countable Markov maps with infinitely many inverse branches and establish the following continuity theorem: Let $T_k$ and $T$ be expanding countable Markov maps such that the inverse branches of $T_k$ converge pointwise to the inverse branches of $T$ as $k \to \infty$. Then under suitable regularity assumptions on the maps $T_k$ and $T$ the following limit exists:
$$\lim_{k \to \infty} \Hd \{x : \theta_k'(x) \neq 0\} = 1,$$
where $\theta_k$ is the topological conjugacy between $T_k$ and $T$ and $\Hd$ stands for the Hausdorff dimension. This is in contrast with the fact that other natural quantities measuring the singularity of $\theta_k$ fail to be continuous in this manner under pointwise convergence such as the H\"{o}lder exponent of $\theta_k$ or the Hausdorff dimension $\Hd (\mu \circ \theta_k)$ for the preimage of the absolutely continuous invariant measure $\mu$ for $T$. As an application we obtain a perturbation theorem in non-uniformly hyperbolic dynamics for conjugacies between intermittent Manneville-Pomeau maps $x \mapsto x + x^{1+\alpha} \mod 1$ when varying the parameter $\alpha$.
\end{abstract}

\maketitle

\sloppy

\section{Introduction and statement of results} \label{sec:intro}

\subsection{Countable Markov maps and singular functions} \label{sec:luroth}

Countable Markov maps, that is, interval maps with countably many expanding branches, have received much attention over the past several years. They appear in particular in Diophantine approximation in the study of approximation rates of irrationals by rational numbers. The key examples here are the \textit{Gauss map} $x \mapsto 1/x \mod 1$, which generates the continued fraction expansion \cite{DK, KS2}, and the various \textit{L\"uroth maps}, which generate L\"uroth expansions \cite{BBDK, KMS, KKK}. Moreover, countable Markov maps appear naturally in the study of non-uniformly hyperbolic dynamical systems such as the intermittent \textit{Manneville-Pomeau maps} \cite{MP}, where often one considers induced countable Markov maps of such systems. Various examples are pictured in Figure \ref{fig:countableMarkovMaps} below.

\begin{figure}[ht!]
\includegraphics[scale=0.58]{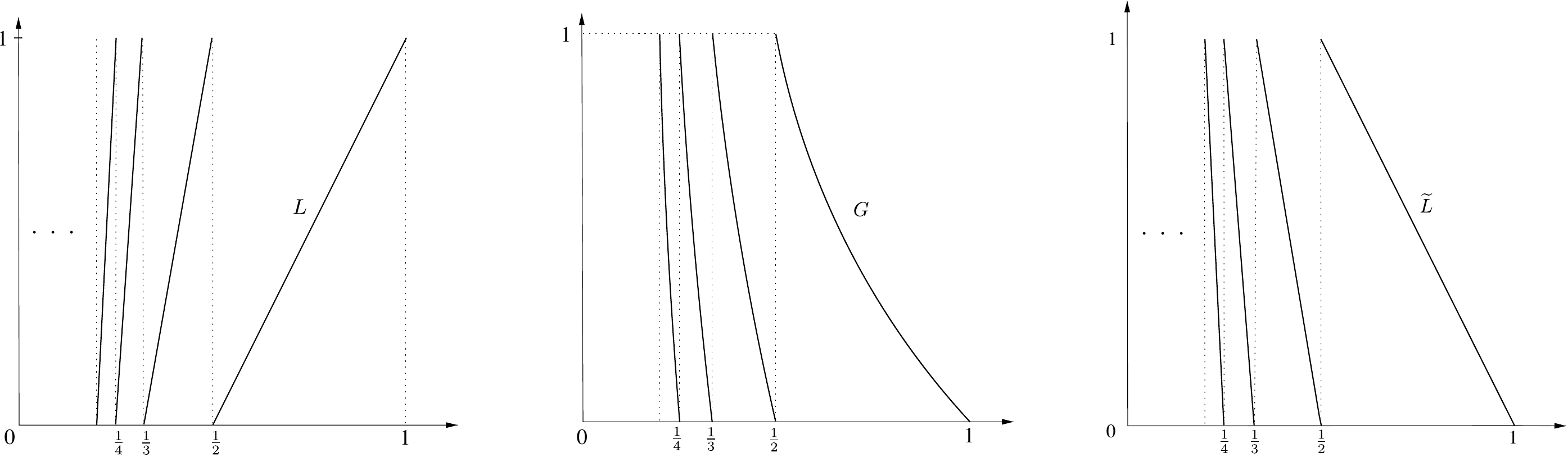}
\caption{Examples of countable Markov maps: the Gauss map $G$, the L\"uroth map $L$ and the alternating L\"uroth map $\tilde L$, see \cite{BBDK} for definitions.}
\label{fig:countableMarkovMaps}
\end{figure}

In this paper, we are interested in the changes to the dynamics of countable Markov maps when  small pointwise perturbations are applied. A possible way to evaluate the effect of such perturbations on the dynamics of these maps is to investigate the topological conjugacies between the original map and the perturbed map, where we recall that a homeomorphism $\theta:(X, T) \to (Y, S)$ between two topological dynamical systems is said to be a {\it topological conjugacy} if $\theta\circ T=S\circ \theta$. In other words, every orbit under $T$ corresponds to an orbit under $S$ and vice versa. In the case of countable Markov maps $T$ and $S$ the conjugacies will usually be strictly increasing, singular maps, otherwise known as slippery Devil's staircases (a term coined by Mandelbrot \cite{Mandelbrot}). Singular here means that the derivative is Lebesgue-almost everywhere equal to zero:
$$\mathrm{Leb}(\{x : \theta'(x) \neq 0\}) = 0.$$
Now the \textit{degree} of the singularity of the conjugacy $\theta$ gives us a certain sense of how ``close'' the maps $T$ and $S$ are. Natural ways to measure the degree of singularity are for example the Hausdorff dimension $\Hd \{x : \theta'(x) \neq 0\}$ or the H\"older exponent of the conjugacy $\theta$.

Perhaps the first well-studied example of a singular function is Minkowski's question-mark function $? :[0, 1]\to [0,1]$, which was constructed by H. Minkowski in 1908 (see \cite{min}). It is illustrated in Figure \ref{fig:minkowskiluroth}. This function was originally designed precisely to map all rational numbers in $[0,1]$ onto the dyadic rationals, and all  algebraic numbers of degree two onto the non-dyadic rationals, in an order preserving way. The main idea was to illustrate the Lagrange property of the algebraic numbers of degree two (see Theorem 28 in \cite{cf}). The function $?$ was proved to be singular by Denjoy \cite{denjoy}, and was also studied by Salem \cite{Salem}.

 More recently, Kesseb\"ohmer and Stratmann \cite{KS2} showed that the Minkowski question-mark function can be thought of as the topological conjugacy between the Gauss map and the alternating L\"uroth map (or, equivalently, between the classical Farey map from elementary number theory and the tent map). Moreover, they showed that the derivative can either take the value zero, be infinite, or else it doesn't exist. They then applied previous thermodynamical results to compute the Hausdorff dimension of the sets where the derivative is infinite and where it doesn't exist, and these dimensions turn out to be equal \cite{KesseboehmerStratmann:07}.

  \begin{figure}[ht!]
\includegraphics[scale=0.7]{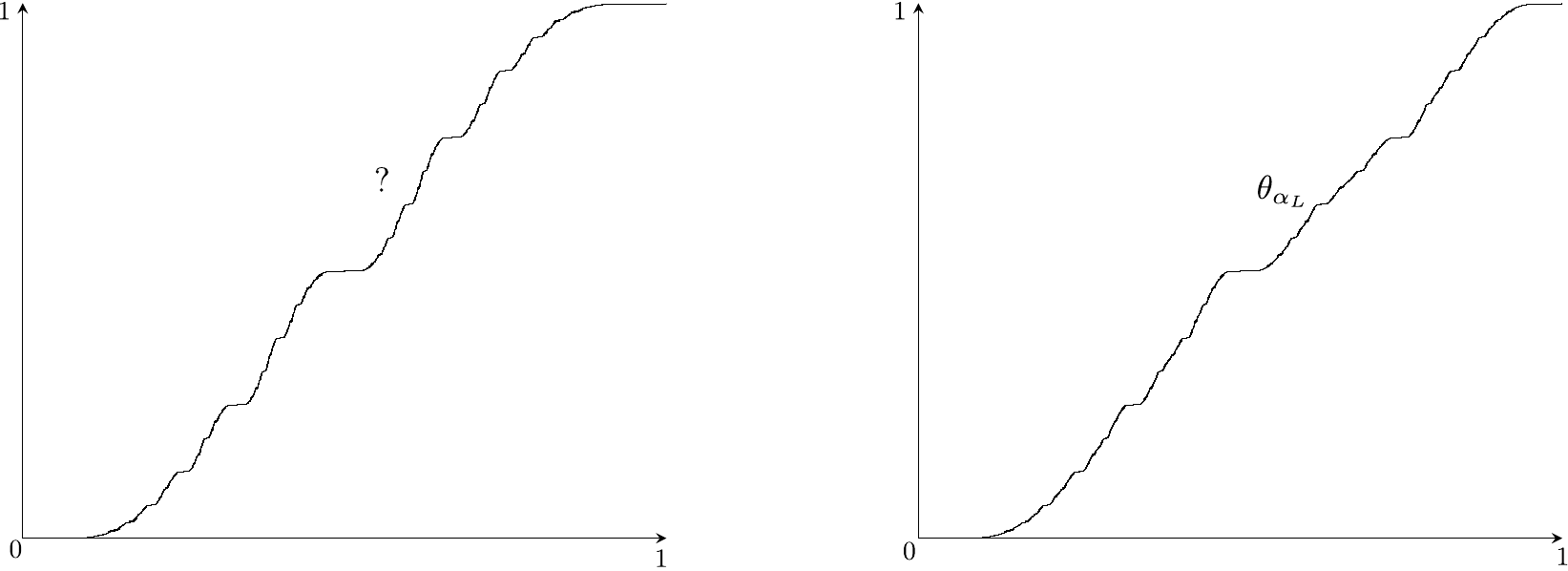}
\caption{Minkowski's question mark function $?$ and the $\alpha_L$-Farey-Minkowski function $\theta_{\alpha_L}$ for the classical alternating L\"uroth map $L$, as shown above.}
\label{fig:minkowskiluroth}
\end{figure}

The Hausdorff dimension of the set of  non-zero derivative for a variant of the Minkowski question-mark function has been studied  by Li, Xiao and Dekking in \cite{LXD}, and for the case of  expanding maps of the interval with finitely many increasing branches by Kesseb\"ohmer {\it et al.} in \cite{JKPS}. A similar problem has also been studied in the case of singular functions which are increasing but not strictly increasing, such as for several variants of the Cantor ternary function, see \cite{Dar,LXD,Fa2,KS2,Tro} for example. Moreover, similar results have been considered for topological conjugacies (called $\alpha$-\textit{Farey-Minkowski functions}) between $\alpha$-L\"uroth maps by Munday \cite{smnew} (an example is shown in Figure \ref{fig:minkowskiluroth}) and later by Arroyo \cite{AA}, where he considers the conjugacy maps between the Gauss map and any $\alpha$-L\"uroth map.

\subsection{Perturbations and stability}\label{sec:perturbations}

There is extensive literature on the perturbations of dynamical systems and their effect on entropy, dimension, and other statistical quantities under both random and deterministic perturbations. In our case we will study the following problem: How do the notions of singularity of the topological conjugacy $\theta$ between countable Markov maps $T$ and $S$ behave when $T$ and $S$ are sufficiently close? Here by ``closeness'' we mean the relatively weak notion that the inverse branches of $T$ and $S$ are pointwise close. 

Heuristically here one would expect that the conjugacies $\theta$ would share the properties of the identity mapping as $\theta$ is pointwise close to the identity. We will find out that for the Hausdorff dimension of the set of $x$ with $\theta'(x) \neq 0$, we do have some continuity under pointwise perturbations (see Theorem \ref{mainthm} below), but under other notions of singularity of $\theta$, such as H\"older exponents or Hausdorff dimension of the $\theta$ image of the absolutely continuous invariant measure, the continuity fails to occur (see Propositions \ref{prop:holder}, \ref{prop:hausdorff}, \ref{prop:lyapunov} below) due to the non-compact nature of countable Markov maps.

To state our main result, let us first fix a little notation (we refer to Section \ref{sec:prelim} for a more thorough exposition). Let $f_i : [0,1] \to [0,1]$ be $C^1$ contractions for each $i\in\N$ and where either $f_{1}(0)=1$, $f_{i+1}(0)=f_{i}(1)$  for all $i\in\N$ and $(f_i(0))$ is a decreasing sequence with $\lim_{i\to\infty}f_i(0)=0$ or we have that $f_1(1)=1$,  $f_{i+1}(1)=f_i(0)$ for all $i\in\N$ and $(f_i(1))_{i \in \N}$ is a decreasing sequence. These maps are the inverse branches of a piecewise differentiable countable Markov map $T$. We assume some regularity on $T$ and a standard assumption in this setting is that the geometric potential $-\log|T'|$ has \textit{summable variations} (see Section \ref{sec:thermo} for a definition), that is,  
$$\sum_{n = 1}^\infty \var_n(-\log |T'|) < \infty,$$ 
which is  satisfied, for example, for the Gauss map, jump transformations of the Manneville-Pomeau map, and for all $\alpha$-L\"uroth maps. We will fix such a system \{$T$, $(f_i)_{i \in \N}$\} and consider perturbations of the system, in the following sense. 

For each $k\in\N$ we will consider a system with maps $f_{i,k}$ and $T_{k}$ satisfying the variation assumption above and where for each $x\in [0,1]$ we have 
$$\lim_{k\to\infty} f_{i,k}(x)=f_i(x).$$ We need that $f_{i,k}$ have the same orientation as the maps $f_i$. This means the dynamical systems $T_{k}$ and $T$ are  topologically conjugate and we will denote the conjugacy by $\theta_{k}$, that is the homeomorphism $\theta_k:[0,1]\to [0,1]$ satisfies that $T\circ\theta_k=\theta_k\circ T_k$. Now the pointwise convergence of the inverse branches guarantee that when $k \to \infty$, we have that the conjugacy $\theta_k$ will flatten and converge pointwise to the identity mapping, see Figure \ref{fig:limitperturbation} for example. 

 \begin{figure}[ht!]
\includegraphics[scale=0.5]{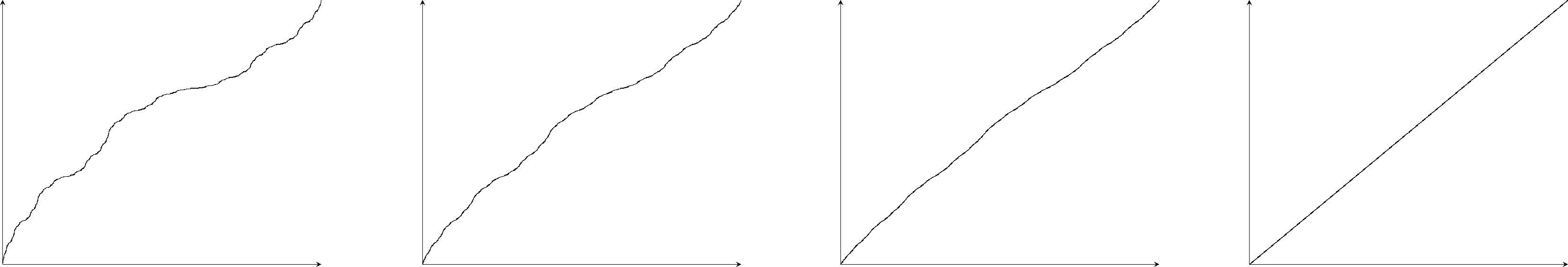}
\caption{Four conjugacies $\theta_k$ between two countable Markov maps $T_k$ and $T$. The map $T$ is the $\alpha_D$-L\"uroth map for the dyadic partition $\alpha_D$ and $T_k$ is the $\alpha$-L\"uroth map for a $\lambda$-adic partition for $\lambda$ attaining the values $3$, $2.5$, $2.1$ and $2$. The maps $\theta_k$ approach the identity pointwise when $f_{i,k} \to f_i$ pointwise.}
\label{fig:limitperturbation}
\end{figure}

Thus one would expect that $\theta_k$ should share the properties of the identity in the limit. Our main result shows that this happens for the Hausdorff dimension of the set $\{x:\theta_{k}'(x)\neq 0\}$ under suitable assumptions on the converging family of countable Markov maps.

\begin{theorem}\label{mainthm}
Suppose $T$ is a countable Markov map with inverse branches $f_i$ such that the potential $-\log |T'|$ has summable variations. Let $(T_k)$ be a sequence of countable Markov maps with inverse branches $f_{i,k}$. Assume the following two assumptions on the tail and variations:
\begin{itemize}
\item[\emph{(1)}] There exists $0 < t<1$ with
$$\sum_{i=1}^{\infty}|f_i[0,1]|^t<\infty.$$
\item[\emph{(2)}] The potentials $-\log |T_k'|$ have summable variations with a uniform bound over $k \in \N$:
$$\sup_{k \in \N} \sum_{n = 1}^\infty \var_n(-\log |T_k'|) < \infty.$$
\end{itemize}
Under these assumptions, if for any $i \in \N$ the inverse branches $f_{i,k} \to f_i$ pointwise as $k \to \infty$, we have
\[
\lim_{k\to\infty}\dim_{\mathrm{H}}\{x:\theta_{k}'(x)\neq 0\}=1.
\]
\end{theorem}

Let us make a few remarks on the conditions (1) and (2) required in Theorem \ref{mainthm}. The condition (1) holds if the countable Markov map $T$ has at most a polynomially fat tail, in the sense that the lengths $|f_i[0,1]| = O(i^{-p})$ as $i \to \infty$ for some $p > 1$. Thus (1) yields in particular that the absolutely continuous invariant measure for $T$ has finite entropy, but it is not an equivalent condition. The condition (2) on variation in Theorem \ref{mainthm} is satisfied if the inverse branches of $T_k$ are linear, i.e., when the maps $T_k$ are $\alpha$-L\"uroth maps for certain partitions $\alpha$ in the notation of \cite{KMS}. Thus our result gives rather general conditions to have such a perturbation theorem for $\alpha$-L\"uroth maps, provided that the map being perturbed has a thin enough tail.

In the non-linear case, the Gauss map will satisfy the tail assumption (1) we impose, so the perturbation theorem is valid provided we have a uniform bound (2) over the sums of variations on the family of maps converging to the Gauss map. Furthermore, the conditions in Theorem \ref{mainthm} are weak enough for us to apply Theorem \ref{mainthm} to the study of a certain family of intermittent maps in non-uniformly hyperbolic dynamics known as the \textit{Manneville-Pomeau maps} $M_\alpha : [0,1] \to [0,1]$,
$$M_\alpha (x) := x + x^{1+\alpha} \mod 1, \quad x \in [0,1],$$
for a parameter $0 < \alpha < \infty$. The jump transformations (in other words, ``accelerated dynamics'' or induced maps) for $M_\alpha$ give us countable Markov maps that have polynomial tail and satisfy the assumptions of Theorem \ref{mainthm} when varying the parameter $\alpha$ for the maps $M_\alpha$, since this means pointwise convergence of the inverse branches.
Thus we obtain the following corollary to Theorem \ref{mainthm}:

\begin{corollary}\label{cor:mp}
Let $\alpha > 0$. Then as $\beta \to \alpha$ we have
$$\Hd \{x : \theta_{M_\beta,M_\alpha}'(x) \neq 0\} \to 1,$$
where $\theta_{M_\beta,M_\alpha}$ is the topological conjugacy between the Manneville-Pomeau maps $M_\beta$ and $M_\alpha$.
\end{corollary}

Corollary \ref{cor:mp} concerns the topological stability for $M_\alpha$ when varying $\alpha$. A related area of study for Manneville-Pomeau maps is the measure theoretical \textit{statistical stability}, where the behaviour of the absolutely continuous invariant measure for $M_\alpha$ is studied when varying $\alpha$, see for example the recent works by Freitas and Todd \cite{FT} and Baladi and Todd \cite{BT}.

There are also other natural ways to measure the singularity of the conjugacies $\theta_k$ and the effect of perturbations to them. However, we will see that the continuity as presented in Theorem \ref{mainthm} fails for these quantities. We will consider three possible examples below.

Firstly, observe that the topological conjugacies $\theta_k$ are all H\"older continuous. Thus one might expect that the H\"older exponent $\kappa(\theta_k)$ of $\theta_k$ (see Section \ref{sec:prelim} for definitions) would converge to $1$, which is the H\"older exponent of the identity. However, this can be made to fail:

\begin{proposition}\label{prop:holder}
There exist examples of $T_k$ and $T$ satisfying the assumptions of Theorem \ref{mainthm} such that the H\"older exponents $\kappa(\theta_k)$ of $\theta_k$ satisfy
$$\lim_{k \to \infty} \kappa(\theta_k) = 0.$$
\end{proposition}

A similar behaviour can be observed also in the following setting. If $\mu$ is the absolutely continuous $T$-invariant measure, then one might also expect that the Hausdorff dimensions $\Hd(\mu \circ \theta_k)$ of the $\theta_k$-preimages of the measure $\mu$ would converge to $1$. On the other hand, the maps $T_k$ can be chosen such that the the dimensions do not converge to the expected value:

\begin{proposition}\label{prop:hausdorff}
There exist examples of $T_k$ and $T$ satisfying the assumptions of Theorem \ref{mainthm} such that the Hausdorff dimensions of $\mu \circ \theta_k$ satisfy
$$\lim_{k \to \infty} \Hd (\mu \circ \theta_k) = 0.$$
\end{proposition}

Moreover, denoting by $\mu_k$ the absolutely continuous $T_k$-invariant measure, we also consider the entropy (that is, the Lyapunov exponent) of the absolutely continuous invariant measures for the maps $T_k$ and $T$ respectively. If we would have that $h(\mu_k,T_k) \to h(\mu,T)$, instead of pointwise convergence of the inverse branches of $T_k$, it would be considerably easier to prove the statement of the main result Theorem \ref{mainthm}. However, $h(\mu_k,T_k) \to h(\mu,T)$ is  too strong a property to be deduced from pointwise convergence, as the following result shows.

\begin{proposition}\label{prop:lyapunov}
There exist examples of $T_k$ and $T$ satisfying the assumptions of Theorem \ref{mainthm} such that the entropy $h(\mu,T) < \infty$ but the limit
$$\lim_{k \to \infty}h(\mu_k,T_k) = \infty.$$
\end{proposition}

We remark that in the uniformly hyperbolic compact case, i.e., in the situation of  finitely many branches with uniform expansion rate, all these notions can be shown to be continuous under pointwise perturbations. The heuristic reason for Propositions \ref{prop:holder}, \ref{prop:hausdorff} and \ref{prop:lyapunov} is that they represent notions that are very sensitive to the tail behaviour of the countable Markov maps $T_k$. On the other hand, the idea of the proof of Theorem \ref{mainthm} is that we approximate the infinite systems considered by a finite branch system and in this approximation the precise nature of the tails is not so important, except in terms of the tail of the limiting map $T$ (the tail condition (1) of Theorem \ref{mainthm}). Thus the Hausdorff dimension of non-differentiability points will not be as sensitive to the tails as the H\"older exponent $\kappa(\theta_k)$, Hausdorff dimension of $\mu \circ \theta_k$ or the entropy $h(\mu_k,T_k)$.

The limit obtained in Theorem \ref{mainthm} does not tell us about the possible rate of the numbers $\dim_{\mathrm{H}}\{x:\theta_{k}'(x)\neq 0\}$ converging to $1$ as $k$ approaches infinity. If we restrict the class of countable Markov maps we consider, then this can be addressed and the Hausdorff dimension can be explicitly computed. For this, we will consider a class of countable Markov maps similar to those arising from the \textit{Salem family} considered in \cite{JKPS}. Fix $0 < \tau < 1$ and define the map $T_\tau$ to be the countable Markov map with decreasing linear branches on each interval $(\tau^k,\tau^{k-1}]$, $k \in \N$. In the language of $\alpha$-L\"uroth maps \cite{KMS}, the map $T_\tau$ is the $\alpha$-L\"uroth map for the partition $\alpha = \{(\tau^k,\tau^{k-1}] : k \in \N\}$. We obtain the following theorem.

\begin{theorem}
\label{mainthm2}
Fix $0 < \tau \neq \tau' < 1$ and let $\theta_{\tau,\tau'}$ be the topological conjugacy between $T_\tau$ and $T_{\tau'}$. Then
$$\Hd \{x:\theta_{\tau,\tau'}'(x)\neq 0\} = \frac{p_{\tau,\tau'} \log p_{\tau,\tau'}+(1-p_{\tau,\tau'})\log(1-p_{\tau,\tau'})}{p_{\tau,\tau'}\log \tau+(1-p_{\tau,\tau'})\log(1-\tau)},$$
where
$$p_{\tau,\tau'} := \frac{\log (1-\tau') - \log(1-\tau)}{\log \tau - \log(1-\tau) - \log \tau' + \log (1-\tau')}.$$
\end{theorem}

Due to the choice of the specific countable Markov maps $T_\tau$, the proof of Theorem \ref{mainthm2} is reduced to the study of conjugacies between tent-like expanding maps with two full branches, one increasing and one decreasing. A similar result was obtained in \cite[Theorem 1.1]{JKPS}, where the authors consider a family of expanding maps with finitely many increasing full branches. However, as we have one increasing and one decreasing branch, the proof  in our situation is rather simpler than in \cite{JKPS}.

\subsection{Organisation of the paper} The paper is organised as follows. In Sections \ref{sec:prelim} and \ref{sec:thermo} we will give all the necessary background results from dimension theory and thermodynamic formalism. In Section \ref{sec:mainthm} we will give the proof of Theorem \ref{mainthm}. In Section \ref{sec:holderentropy} we present how to achieve Propositions \ref{prop:holder}, \ref{prop:hausdorff} and \ref{prop:lyapunov}. In Section \ref{sec:manpom} we discuss the Manneville-Pomeau example further and prove Corollary \ref{cor:mp}, and finally, in Section \ref{sec:geometric} we prove Theorem \ref{mainthm2}.

\section{Preliminaries and notations}\label{sec:prelim}

\subsection{Interval maps and modeling with $\N^\N$} A \textit{countable Markov map} $T : [0,1] \to [0,1]$ is defined with the help of its inverse branches. We consider the situation where for each $i\in\N$, there exist maps $f_i:[0,1]\to [0,1]$ which are continuous and strictly decreasing on $[0,1]$ and differentiable on $(0,1)$. We further assume that there exists $m\in\N$ and $\xi<1$ such that for all $(i_1,\ldots,i_m)\in\N^m$ we have that $|(f_{i_1}\circ\cdots\circ f_{i_m})'(x)|\leq\xi$ for all $x\in (0,1)$. We will also suppose that $f_1(0)=1$, $f_{i}(1)=f_{i+1}(0)$ for all $i\in\N$ and $\lim_{i\to\infty}f_i(0)=0$ or alternatively that $f_1(0)=0$, $f_{i}(0)=f_{i+1}(1)$ for all $i\in\N$ and $\lim_{i\to\infty}f_i(0)=0$. Thus $\bigcup_{i=1}^{\infty} f_i([0,1])=(0,1]$ and if $i\neq j$ then $f_i((0,1))\cap f_j((0,1))=\emptyset$. We define an expanding map $T: [0,1]\to [0,1]$ by setting
$$T(x):=\left\{
          \begin{array}{ll}
            f_i^{-1}(x), & \hbox{if $x\in f_i([0,1))$;} \\
            0, & \hbox{if $x=0$.}
          \end{array}
        \right.
$$

Given a countable Markov map $T$ with inverse branches $f_i$, $i\in\N$, it is convenient to model our systems using symbolic dynamics.
Let $\Sigma := \N^\N$ and let $\sigma : \Sigma \to \Sigma$ be the usual left-shift transformation. We can relate this to our systems $\{f_i\},T$ via projections $\pi_T : \Sigma \to [0,1]$. We define
$$\pi_T(i_1,i_2,\dots) := \lim_{n \to \infty} f_{i_1}\circ f_{i_2} \circ \dots \circ f_{i_n}(0)$$
The factor map $\pi_T$ allow us to import the thermodynamical formalism from the shift space to measures invariant under $T$. For a shift invariant measure $\mu$, the push-forward measure $\pi_T \mu := \mu \circ \pi_T^{-1}$ will be $T$-invariant. Moreover if $\mu$ is ergodic for the shift map then $\pi_T \mu$ will be ergodic for $T$. Thus we can use  the symbolic model $(\Sigma,\sigma)$ and the geometric model $([0,1],T)$ interchangably.

Now if we have a sequence of countable Markov maps $T_k$ with inverse branches $\{f_{i,k}\}$ satisfying the assumptions of Theorem \ref{mainthm}, we will shorten the notation by letting $\pi_k := \pi_{T_k}$ and $\pi := \pi_T$. Then the topological conjugacy $\theta_k$ between $T_k$ and $T$ will satisfy
$$\theta_k(x)=\pi\circ\pi_k^{-1}(x), \quad x \in [0,1].$$
In other words, the conjugacy map between the systems $T$ and $T_k$ takes the point $x$ with coding given by $T$ and sends it to the point with the same coding, but now understood in terms of $T_k$.

\subsection{Dimension and H\"older/Lyapunov exponents}

Let $\Hd A$ be the Hausdorff dimension of a set $A \subset \R$ and the $s$-dimensional Hausdorff measures $\cH^s$ and the $\delta$-Hausdorff content $\cH^s_\delta$, see \cite{Fal} for a definition. For a Radon measure $\nu$ on $\R$, the Hausdorff dimension of $\nu$ is defined to be
$$\Hd \nu := \inf\{\Hd A : \nu(A) > 0\} = \essinf_{x \sim \nu} \llocd(\nu,x),$$
where $\llocd(\nu,x)$ is the \textit{lower} local dimension of $\nu$ at $x$, which is defined by
$$\llocd(\nu,x) := \liminf_{r \searrow 0} \frac{\log \nu(B(x,r))}{\log r}.$$

\begin{definition}[H\"older exponent] If $\theta : [0,1] \to [0,1]$ is a function, then the \textit{H\"older exponent} $\kappa(\theta)$ of $\theta$ is defined to be the infimal $\kappa \geq 0$ such that for some $C > 0$ the following inequality holds:
$$|\theta(x)-\theta(y)| \leq C|x-y|^\kappa, \quad x,y \in [0,1].$$
\end{definition}

Now we will consider a fixed measure $\mu$ on $[0,1]$ and countable Markov map $T$ and we will define the notions of Lyapunov exponents and entropy for this measure. Note that the Lyapunov exponent depends upon the mapping $T$ as well as the measure $\mu$.

\begin{definition}[Lyapunov exponent] The {\em Lyapunov exponent} of the measure $\mu$ is defined to be
$$\lambda(\mu,T) := \int \log |T' | \,\d \mu.$$
\end{definition}

Similarly, if $I_\i^T = \pi_T[\i]$, for $\i \in \N^*$, are the construction intervals generated by the countable Markov map $T$, the entropy of $\mu$ is defined as follows:

\begin{definition}[Entropy] The {\em Kolmogorov-Sinai entropy} (with respect to $T$) of the measure $\mu$ is defined to be
$$ h(\mu,T) := \lim_{n\to\infty} \frac{1}{n}\sum_{\i \in \N^n} -\mu(I_\i^T)\log \mu (I_\i^T).$$
\end{definition}

Note that sometimes we also write $h(\mu,T)$ or $\lambda(\mu,T)$ for a measure $\mu$ living on $\Sigma$ and then we just mean the values $h(\pi_T \mu,T)$ and $\lambda(\pi_T \mu,T)$ respectively for the projected measure $\pi_T \mu$. If we just take the entropy of such $\mu$ with respect to the shift map $\sigma$ on $\Sigma$, we define $h(\mu,\sigma)$ like $h(\mu,T)$ but we replace the intervals $I_\i^T$ by the cylinders $[\i]$.

Now, given a countable Markov map $T$, the Hausdorff dimensions of each of the $\pi_T$-projections of an ergodic shift-invariant measure can be computed using the following result:
\begin{proposition}[Mauldin-Urba\'{n}ski]\label{prop:measdim}
If $\mu$ is an ergodic $T$ invariant probability measure on $[0,1]$ and $h(\mu,T) < \infty$, then the Hausdorff dimension of $\mu$ is given by
$$\Hd \mu=\frac{h(\mu,T)}{\lambda(\mu,T)}.$$
\end{proposition}

The above result can be found as Theorem 4.4.2 in the book \cite{MU} by Mauldin and Urba\'{n}ski.

\section{Thermodynamical formalism for the countable Markov shift}\label{sec:thermo}

In this section we present the tools we will need from thermodynamical formalism. We mostly concentrate on the countable Markov shift $\Sigma$ as this is where we will reformulate the problem, using the theory developed in a much more general setting in D. Mauldin and M. Urba\'nski \cite{MU} and the series of works by O. Sarig, see for example \cite{sarigETDS, sarigPAMS}.

First, recall that a potential $\phi$ is said to be \textit{locally H\"{o}lder} if there exist constants $C>0$ and $\delta \in (0,1)$ such that for all $n\in\N$ the \textit{variations} $\mathrm{var}_n$ decay exponentially:
$$\mathrm{var}_n(\phi):=\sup_{\i \in \N^n} \{|\phi(\j)-\phi(\k)| : \j,\k \in [\i]\} \leq C\delta^n.$$
Note that since nothing is assumed in the case that $n=0$, this does \textit{not} imply that $\phi$ is bounded.

The \textit{Birkhoff sum} $S_n \phi$ of a potential $\phi : \Sigma \to \R$ is the potential  defined by
$$S_n \phi(\i) := \sum_{k = 0}^{n-1} \phi(\sigma^k(\i)).$$
The \textit{pressure} of a locally H\"{o}lder potential $\phi$ is then the limit
$$P(\phi) := \lim_{n \to \infty} \frac{1}{n} \log \left(\sum_{\i \in \N^n} \exp(S_n\phi(\i^\infty))\right),$$
where $\i^\infty = \i\i\i\dots$ is the periodic word repeating the word $\i \in \N^n$. Define $\cM_\sigma$ to be the collection of all $\sigma$-invariant measures on $\Sigma$. A deep and useful result which we will now state is the \textit{variational principle}, which gives a representation of $P(\phi)$  using the Kolmogorov-Sinai entropy:

\begin{lemma}[Variational principle]\label{lma:varprinciple}
For any locally H\"{o}lder potential $\phi$  we have that
$$P(\phi) = \sup_{\mu \in \cM_\sigma} \left\{ h(\mu,\sigma) + \int \phi \, d\mu: \int \phi \, d\mu > -\infty \right\}.$$
\end{lemma}
For a proof, see Theorem 2.1.8 in \cite{MU}. If there exists a measure $\mu \in \cM_\sigma$ which attains the supremum in Lemma \ref{lma:varprinciple}, then we call $\mu$ an \textit{equilibrium state} for a potential $\phi$. In the case of finite pressure more can be said about equilibrium states.
\begin{definition}[Gibbs measures]
Let $\phi:\Sigma\to\R$ be a locally H\"{o}lder potential. If $P(\phi)$ is finite, then we call $\mu_{\phi}$ a \textit{Gibbs measure} for $\phi$ if there exists a constant $C > 0$ such that
$$C^{-1}\exp(S_n \phi(\j) - nP(\phi)) \leq \mu_\phi[\i] \leq C\exp(S_n \phi(\j) - nP(\phi))$$
for any $\i \in \N^n$, $\j \in [\i]$ and $n \in \N$.
\end{definition}
An example of such a measure is the \textit{Bernoulli measure} $\mu$ associated to weights $p_i \in [0,1]$, $i \in \N$, with $\sum_{i = 1}^\infty p_i = 1$, which is the equilibrium state for the potential $\phi(\i) = -\log p_{i_1}$. Then $P(\phi) = 0$ and
$$\mu[\i] = p_{i_1}\dots p_{i_n} = \exp(S_n \phi(\j)), \quad \text{for }\j \in [\i].$$
The following proposition relates Gibbs measures to equilibrium states.
\begin{proposition}
\label{prop:uniqequilibrium}
Let $\phi:\Sigma\to\R$ be a locally H\"{o}lder potential. If $P(\phi)<\infty$ then there exists a unique invariant probability measure, $\mu_{\phi}$ which is a Gibbs measure for $\phi$. Moreover, if $\phi$ is integrable with respect to $\mu_{\phi}$ then $\mu_{\phi}$ is the unique equilibrium state for $\phi$.
\end{proposition}
For a proof  of this result, see Proposition 2.1.9, Theorem 2.2.9 and Corollary 2.7.5 in \cite{MU}. The case when $\phi$ is not integrable with respect to $\mu_\phi$ is the subject of the next lemma.
\begin{lemma}\label{lem:noequilibrium}
Let $\phi:\Sigma\to\R$ be a locally H\"{o}lder potential with $P(\phi) < \infty$. If $\phi$ is not $\mu_{\phi}$ integrable, then there exist no equilibrium states for $\phi$.
\end{lemma}

\begin{proof}It is a result of Sarig \cite[Theorem 7]{sarigETDS} that the only possible equilibrium state is a fixed point for the Ruelle operator (see \cite{sarigETDS} for a definition). It is then shown in the proof of \cite[Theorem 1]{sarigPAMS} that in the situation where the system satisfies the  Big Image Property (see Sarig's paper for the definition; note that it includes the full shift) such measures are Gibbs measures. Thus there cannot exist equilibrium states for $\phi$.
\end{proof}

All the above thermodynamic definitions can be formulated also for the finite alphabet $\{1,2,\dots,N\}$, $N\in \N$ and it makes things considerably simpler. For instance, in the finite alphabet case it is known that unique equilibrium states always exist for H\"{o}lder potentials and they are Gibbs measures. This makes it convenient to restrict to the finite case and consider approximations for the pressure. Given a locally H\"{o}lder potential $\phi : \Sigma \to \R$, we write $P_N(\phi)$ to denote the pressure of $\phi$ restricted to the finite shift $\Sigma_N := \{1,2,\dots,N\}^\N$. Then we have the following approximation result, which can be found as Theorem 2.1.5 in \cite{MU}.
\begin{theorem}[Finite approximation property]\label{finiteapprox}
For any locally H\"{o}lder potential $\phi$,
$$P(\phi) = \lim_{N \to \infty} P_N(\phi).$$
\end{theorem}

This theorem will allow us to use  results which hold on the full shift with a finite alphabet (or, more generally, on topologically mixing subshifts of finite type). These results can sometimes be extended to the infinite case, but due to the hypotheses needed it is more convenient to use Theorem \ref{finiteapprox} and the results in the finite alphabet case.
The first of these results that we will need is the following lemma on the derivative of pressure, which is Proposition 4.10 in \cite{PP}.

\begin{lemma}[Derivative of pressure]\label{lem:thermoderivative}
Let $\phi,\psi:\Sigma_N\to\R$ be H\"{o}lder continuous functions and define the analytic function
$$Z_N(q) := P(q\psi+\phi).$$
Let $\mu_{q}$ be the Gibbs measure on $\Sigma_N$ for the potential $q\psi+\phi$. Then the derivative of $Z_N$ is given by
$$Z_N'(q) = \int \psi \d\mu_q.$$
\end{lemma}

Gibbs measures satisfy many statistical theorems similar to ones in probability theory. We will use one of these,  namely, the \textit{law of the iterated logarithm}. Before stating this theorem, we recall that a function $\psi : \Sigma_N \to \R$ is said to be \textit{cohomologous to a constant} if there exists a constant $c \geq 0$ and a continuous function $u:\Sigma_N\to\R$ such that
$$\psi - c = u - u \circ \sigma.$$
Moreover, $\psi$ is called a \textit{coboundary} if the constant $c$ is equal to $ 0$.
\begin{lemma}[Law of the iterated logarithm]\label{lem:lawofiterated}
Let $\phi,\psi:\Sigma_N\to\R$ be H\"{o}lder potentials where $\psi$ is not cohomologous to a constant. Then there exists $c(\psi) > 0$ such that for $\mu_\phi$-almost every $x$, we have
$$\limsup_{n \to \infty} \frac{S_n\psi(x)-n\int \psi\d\mu_\phi}{\sqrt{n \log\log n}} = c(\psi).$$
\end{lemma}
\begin{proof}
This is Corollary 2 in \cite{DP}. Note that
$$c(\psi)= \lim_{n \to \infty}\frac{1}{n}\int (S_n\psi-\int \psi\d\mu_{\phi})^2\d\mu_{\phi}$$
and it is shown in Proposition 4.12 of \cite{PP} that $c(\psi)\geq 0$, with equality if and only if $\psi$ is cohomologous to a constant. The number $c(\psi)$ is the variance of $\psi$ with respect to $\mu_{\phi}$ and is also the second derivative of the pressure function $q\to P(q\phi+\psi)$ at $q=0$.
\end{proof}

Finally in this section we need the following result in the countable case regarding the behaviour of equilibrium states.
\begin{lemma}\label{cusp}
Let $\phi:\Sigma\to(-\infty,0]$ be locally H\"{o}lder such that $P(\phi)=0$, and let
$$s=\inf\{t:P(t\phi)=\infty\}<\infty.$$
We have that
\begin{enumerate}
\item[\emph{(1)}]
there exists a sequence $\mu_n$ of compactly supported $\sigma$-invariant ergodic measures such that
$$\lim_{n\to\infty}h(\mu_n,\sigma)=\infty \quad \text{and} \quad \limsup_{n\to\infty}\frac{h(\mu_n,\sigma)}{\int \phi\d\mu_n}\geq s,$$
\item[\emph{(2)}]
for any $t>s$ there exists $K(t) > 0$ such that if $\mu$ is ergodic, $\phi$ is integrable with respect to $\mu$ and $h(\mu,\sigma)>K(t)$, then
$$h(\mu,\sigma)+t\int\phi\d\mu< 0.$$
\end{enumerate}
\end{lemma}
\begin{proof}
Let $\epsilon>0$. We can always find $t\geq \max\{0,s-\epsilon\}$ such that $P(t\phi)=\infty$. Therefore we can find $N\in\N$ such that
$$P_N(t\phi)\geq \max\{P((s+\epsilon)\phi)+2,0\}\geq P_N((s+\epsilon)\phi).$$
Let $z:\R\to\R$ be defined by $z(r)=P_N(r\phi)$, and observe that $z(t)\geq 0$. Also, by the mean value theorem and the convexity of pressure, $z'(t) \leq -1/\epsilon$. By Lemma \ref{lem:thermoderivative} the equilibrium state $\mu$ on $\Sigma_N$ for $t\phi$ will satisfy that $\int \phi\d\mu\leq -1/\epsilon$ and $\frac{h(\mu,\sigma)}{\int \phi\d\mu}\geq t$. To complete the proof of the first part for each $n\in\N$ simply take $\epsilon=1/n$ to find the sequence of measures $\mu_n$.

Now let $t>t_1>s$. Thus $P(t_1\phi)<\infty$ and so, by the variational principle, for any ergodic measure $\mu$ for which $\phi$ is integrable we have
$$t_1\int\phi\d\mu+h(\mu,\sigma)\leq P(t_1\phi)<\infty$$
and since, by assumption, $P(\phi)=0$ we have that $h(\mu,\sigma)\leq-\int\phi\d\mu$.
Thus if $h(\mu,\sigma)\geq-t\int\phi\d\mu$ then
$$-t\int\phi\d\mu+t_1\int\phi\d\mu\leq P(t_1\phi).$$
Thus
$$h(\mu,\sigma)\leq-\int\phi\d\mu\leq \frac{P(t_1\phi)}{t-t_1}.$$
In other words, taking the contrapositive, we have that if $h(\mu,\sigma)>\frac{P(t_1\phi)}{t-t_1}$ then $h(\mu,\sigma)+t\int\phi\d\mu< 0$, and the proof is complete.
\end{proof}

\section{Proof of the main theorem}\label{sec:mainthm}

In this section we will present the proof of Theorem \ref{mainthm}. To this end, fix the countable Markov maps $T_k$ and $T$ and define the potentials
$$\phi_k(\i) := -\log|T_k'(\pi_k(\i))| \quad \text{and}\quad \phi(\i) := -\log|T'(\pi(\i))|$$
for $\i \in \Sigma$. Recall that by the assumption Theorem \ref{mainthm}(2) these potentials have uniformly bounded sums of variations. Our first step is to slightly simplify the problem by `iterating' these potentials to a suitable generation $m \in \N$ such that the distortion of $\phi_k$ and $\phi$ from analogous potentials coming from systems with linear branches is small. This is possible due to the bounded variations.

For this purpose, let us fix a generation $m \in \N$ and denote by $f_{\i,k}$ for $\i \in \N^m$ the inverse branch corresponding to $\i$ of the $m$-fold composition map $T_k^m = T_k \circ T_k \circ \dots \circ T_k$. We define the branches $f_\i$ similarly for the map $T^m$. Now these maps determine intervals
$$I_{\i,k} := f_{\i,k}([0,1]) \quad \text{and} \quad I_\i := f_\i([0,1]).$$
We denote the lengths of these intervals by $a_{\i,k}$ and $a_\i$ respectively.

To bound the Hausdorff dimension of the set $\{x : \theta_k'(x) \neq 0\}$ of non-zero derivative for some $k \in \N$, we must find a \textit{compactly supported} ergodic measure $\mu$ on the shift space $\N^\N$ for which the $\pi_k$ projection of typical points will not have a derivative. Moreover, we will aim to choose the measure $\mu$ such that its Hausdorff dimension is close to $1$ when $k$ is large. This will be done in the following steps:
\begin{itemize}
\item[(1)] In Lemma \ref{lma:step1} we will first iterate the potentials $\phi_k$ and $\phi$ to the $m$-th generation (for some large $m \in \N$) by studying the potentials $\psi_k := \tfrac{1}{m}S_m\phi_k$ and $\psi := \tfrac{1}{m}S_m\phi$ and then use the absolutely continuous and invariant measure for $T$ to construct a $\sigma^m$ Bernoulli measure $\mu_k^m$ on $\N^\N$ which satisfies both that $-\int \psi_k \d\mu^m_k>-\int \psi \d\mu^m_k$ and that the $\pi_k$ projection of $\mu_k^m$ has dimension close to $1$. The construction is possible due to the pointwise convergence of the inverse branches and the tail/variation assumptions in Theorem \ref{mainthm}.
\item[(2)] The measure $\mu_k^m$ induces  canonically a $\sigma$-invariant measure $\eta = \frac{1}{m}\sum_{i = 0}^{m-1} \sigma^i \mu_k^m$ of the same dimension as $\mu_k^m$ for which $\int \phi_k \, d\eta > \int \phi \, d\eta$. The measure $\eta$ allows us to apply thermodynamic formalism (Lemmas \ref{lem:pressure} and \ref{lem:pressurefinite}) and invoke finite approximation properties (Lemma \ref{step2}) to find a \textit{compactly supported} Gibbs measure $\mu$ where $\int \phi_k\d\mu = \int \phi\d\mu$ but $\phi_k - \phi$ is not a coboundary, and $\mu$ still has dimension close to $1$.
\item[(3)] We will then  essentially apply the law of iterated logarithms (Lemma \ref{lil}) and the coboundary condition to show that for typical points under the projection of the measure $\mu$ the derivative of $\theta_k$ does not exist and the dimension of the projection of this measure will be a lower bound for the dimension of the set of points with non-zero derivative. We then show that  this dimension tends to $1$ as $k$ tends to infinity, which completes the proof.
\end{itemize}

Let us begin by constructing the Bernoulli measure $\mu_k^m$.

\begin{lem}\label{lma:step1}
For each $0 < \delta < 1/3$ there exists $M(\delta)\in\N$ such that for any $m \geq M(\delta)$ there exists $K(m) \in \N$ such that for any $k \geq K(m)$ there exists a $\sigma^m$ ergodic measure $\mu_k^m$ on $\Sigma$ which satisfies
$$-\int S_m\phi_k \d\mu_k^m>-\int S_m\phi \d\mu_k^m \quad \text{and} \quad \Hd  \pi_{k}\mu^m_k= \frac{h(\mu_k^m,T^m)}{-\int S_m\phi_k \d\mu_k^m} \geq \frac{1-3\delta}{1+3\delta}.$$
\end{lem}

For the proof of  Lemma \ref{lma:step1}, we will need the following two preliminary lemmas. We will let $\mu_\phi$ be the equilibrium state for $\phi:\Sigma\to\R$ (and also recall that $\phi(\i)=\log |f_{i_1}'(\pi(\sigma(\i)))|=-\log |T'(\pi(\i))|$).
Since $P(\phi)=0$ we have that $h(\mu_\phi,T)=-\int \phi\d\mu_\phi$.
Let us define the following quantities related to the entropy and Lyapunov exponents. For $m \in \N$,  $\i \in \N^*$ and a potential $f$, let us write
$$\overline{\lambda}_m(f,\i) := \sup\{-S_m f(\j):\j\in [\i]\}$$
and
$$\underline{\lambda}_m(f,\i) := \inf\{-S_m f(\j):\j\in [\i]\}.$$
For the potential $\phi = -\log |T'|$, define the numbers
$$\lambda_m := \sum_{\i \in \N^m} \mu_\phi(I_\i) \overline{\lambda}_m(\phi,\i).$$

\begin{lem}
\label{lma:approximation}
Under the assumptions of Theorem \ref{mainthm}, we have the following approximations
\begin{itemize}
\item[\emph{(1)}] The entropy of the measure $\mu_\phi$ is given by
$$h(\mu_\phi,T) = \lim_{m \to \infty} \frac{1}{m} \lambda_m.$$
\item[\emph{(2)}] There exists $C_0 > 0$ such that for any $m \in \N$ and $\i \in \N^m$ we have
 $$\limsup_{k \to \infty}|\underline{\lambda}_m(\phi_k,\i) - \underline{\lambda}_m(\phi,\i)| \leq C_0.$$
\end{itemize}
\end{lem}

\begin{proof}
(1) By the definition of $\lambda_m$ we have that
$$0\leq-\int S_m\phi\d\mu_\phi\leq\lambda_m\leq -\int S_m\phi\d\mu_\phi+\sum_{k=1}^{\infty}\var_k(\phi).$$
The result then follows since
$$m^{-1}\int S_m\phi\d\mu_\phi=\int\phi\d\mu_{\phi}\text{ and }h(\mu_\phi,T)=-\int\phi\d\mu_{\phi}.$$

(2) Fix $m \in \N$ and $\i \in \N^m$. Let us first verify that
$$\lim_{k \to \infty} f_{\i,k}(y) = f_\i(y)$$
for any $y \in [0,1]$. We will proceed by induction. For $m = 1$, this is the pointwise convergence assumption for the inverse branches of $T_k$ and $T$. Now suppose the claim holds for $m-1$ with $m \geq 2$. Fix $\i \in \N^m$. By the mean value theorem,  there exists a point $z \in [0,1]$ on the interval where the derivative $|f_{i_1,k}'(z)| \leq 1$. Since, according to assumption (2) for Theorem \ref{mainthm}, we have $C:= \sup_{k \in \N} \sum_{n = 1}^\infty \var_n(-\log |T_k'|) < \infty$, this yields that $\|f_{i_1,k}'\|_\infty \leq e^C$ for all $\i \in \N^m$ and $k \in \N$. The mean value theorem gives
$$|f_{i_1,k}(f_{\sigma \i,k}(y)) -  f_{i_1,k}(f_{\sigma \i}(y))| \leq e^C |f_{\sigma \i,k}(y) - f_{\sigma \i}(y)|,$$
which decays to $0$ as $k \to \infty$ by the induction assumption for $m-1$. This completes the proof as
\begin{align*}|f_{\i,k}(y) - f_\i(y)| \leq |f_{i_1,k}(f_{\sigma \i,k}(y)) -  f_{i_1,k}(f_{\sigma \i}(y))| + |f_{i_1,k}(f_{\sigma \i}(y)) -  f_{i_1}(f_{\sigma \i}(y))|
\end{align*}
and the second term on the right-hand side converges to $0$ as $k \to \infty$ by our assumption on pointwise convergence of inverse branches.

Choose $y_k,y \in [0,1]$ such that
$$f_{ \i,k}'(y_k) = f_{ \i,k}(1) - f_{ \i,k}(0) \quad \text{and} \quad f_{ \i}'(y) = f_{ \i}(1) - f_{ \i}(0).$$
This is possible by using the mean value theorem again. Then, by what we proved above, we have that the derivatives $f_{ \i}'(y_k) \to f_{ \i}'(y)$ as $k \to \infty$. Let $\v_k,\v \in [\i]$ be words such that
$$\pi_k(\v_k) = f_{ \i,k}(y_k) \quad \text{and} \quad \pi(\v) = f_{ \i}(y).$$
Then by the chain rule
$$|S_m \phi_k (\v_k) - S_m \phi (\v)| = \big|\log |f_{ \i,k}'(y_k)| - \log |f_{ \i}'(y)|\big|,$$
which converges to $0$ as $k \to \infty$. On the other hand, for any pair $\j,\k \in [\i]$ we have by the triangle inequality
\begin{align*}|S_m \phi_k (\j) - S_m \phi (\k)|  \leq \sum_{\ell = 1}^m \var_\ell(\phi_k) +  | S_m \phi_k (\v_k) - S_m \phi (\v)|  + \sum_{\ell = 1}^m \var_\ell(\phi).
\end{align*}
This yields the claim since $\phi_k$ and $\phi$ have summable variations and by the assumption (2) of Theorem \ref{mainthm} the sums for $\sum_{\ell = 1}^\infty \var_\ell(\phi_k)$ are uniformly bounded over $k \in \N$.
\end{proof}

Let us now make the choice of $M(\delta)$ for a fixed $0 < \delta < 1$:  Write
\begin{equation}\label{Cdef}C := \sum_{m = 1}^\infty \var_m(\phi) + \sup_{k \in \N} \sum_{m = 1}^\infty \var_m(\phi_k) < \infty.\end{equation}
Since by Lemma \ref{lma:approximation} we have $\tfrac{1}{m} \lambda_m \to h(\mu_\phi,\sigma) > 0$, we may choose $M(\delta) \in \N$ such that for any $m \geq M(\delta)$ we have the following properties
\begin{enumerate}[label=(\alph*)]
\item \label{M1}
$$\delta \lambda_m > \max\{C_0,2C\}$$
\item \label{M2}
$$  (1+\delta)\lambda_m  + C \leq mh(\mu_\phi,\sigma)(1+2\delta) ,$$
\item \label{M3}
$$  -\sum\limits_{\i \in \N^m}\mu_{\phi}([\i])\log \mu_{\phi}([\i])\geq mh(\mu_{\phi},\sigma)(1-\delta),$$
\item \label{M4}
$$  -S_m\phi(\j)\geq 1\text{ for all }\j\in \Sigma.$$
\end{enumerate}
where $C_0 > 0$ is the constant from Lemma \ref{lma:approximation}(2), and (d) follows from the assumption on the Markov map $T$ that there exists $m\in\N$ and $\xi<1$ such that for all $(i_1,\ldots,i_m)\in\N^m$ we have that $|(f_{i_1}\circ\cdots\circ f_{i_m})'(x)|\leq\xi$ for all $x\in (0,1)$.

\begin{lem}\label{lma:step1bernoulli}
For each $\delta\in(0, 1/3)$,  we have that either,
\begin{itemize}
\item[\emph{(1)}]
$$-\int \phi\, d\mu_{\phi}<
-\int \phi_k\, d\mu_{\phi}\leq -(1+2\delta)\int \phi\, d\mu_{\phi}, \ \text{ or,}
$$

\item[\emph{(2)}]
For each $m \geq M(\delta)$ and $k \in \N$ there exists a probability vector $(p_{\i,k})_{\i \in\N^m}$ and numbers $r_1(k),r_2(k),r_3(k) \in \R$ satisfying $\lim_{k\to\infty}r_i(k)= 0$ for each $i = 1,2,3$ and such that
\begin{itemize}
\item[\emph{(i)}] $$\sum\limits_{\i \in \N^m} p_{\i,k}\underline{\lambda}_m(\phi_{k},\i)= (1+\delta)\lambda_m+r_1(k);$$
\item[\emph{(ii)}] $$-\sum\limits_{\i \in \N^m} p_{\i,k}\log p_{\i,k} =-\sum \limits_{\i \in \N^m} \mu_{\phi}([\i])\log\mu_{\phi}([\i])+r_2(k);$$
\item[\emph{(iii)}] $$\sum\limits_{\i \in \N^m} p_{\i,k}\overline{\lambda}_m(\phi,\i) = \lambda_m + r_3(k).$$
\end{itemize}
\end{itemize}
\end{lem}

\begin{proof}
Since the measure $\mu_{\phi}$ is not an equilibrium state for $\phi_k$, we have
$$-\int \phi_k\, d\mu_{\phi}>-\int \phi\, d\mu_{\phi} = h(\mu_{\phi},\sigma)$$
and so if  case (1) does not hold, we may assume that
$$-\int \phi_k\, d\mu_{\phi}>-(1+2\delta)\int \phi\, d\mu_{\phi},$$
which yields
$$-m\int S_m\phi_k\, d\mu_{\phi}>-(1+2\delta)m\int S_m \phi\, d\mu_{\phi},$$
by the $\sigma$ invariance of $\mu_\phi$. We put an order on the set of $m$-tuples $\N^m = \{\i(1),\i(2),\dots\}$ by requiring that $\mu_\phi([\i(n)]) \geq \mu_\phi([\i(n+1)])$ and if $\mu_\phi([\i(n)]) = \mu_\phi([\i(n+1)])$ we require that the interval $I_{\i(n)}$ is on the right-hand side of $I_{\i(n+1)}$ (recall that these were obtained as a $\pi = \pi_T$ projection of cylinders onto $[0,1]$). For a fixed $m \geq M(\delta)$ and each $k\in\N$ we define
$$N_k = N_k(m) := \inf\left\{N\in\N:\sum_{n = 1}^N \mu_\phi([\i(n)])\underline{\lambda}_m(\phi_k,\i(n))\geq (1+\delta)\lambda_m\right\}.$$
Note that $N_k$ cannot be infinite since by the choice of $M(\delta)$ (choice \ref{M1}) and by the definition of variations (recall that $C$ is the supremum for the sums of variations of both $\phi_k$ and $\phi$), and the definition of $\lambda_m$ yields
\begin{align*}\sum_{n = 1}^\infty \mu_\phi([\i(n)])\underline{\lambda}_m(\phi_k,\i(n)) & \geq \sum_{n = 1}^\infty \mu_\phi([\i(n)])\overline{\lambda}_m(\phi_k,\i(n)) - C \\
& \geq \int -S_m\phi_k \, d\mu_\phi - C \\
& \geq (1+2\delta) \int -S_m\phi \, d\mu_\phi - C \\
& \geq (1+2\delta)\sum_{n = 1}^\infty \mu_\phi([\i(n)])\overline{\lambda}_m(\phi,\i(n)) - C \\
& \geq (1+2\delta)\lambda_m - 2C \\
& >(1+\delta)\lambda_m.\end{align*}
Our first claim is that $N_k \to \infty$ as $k \to \infty$. This is proved by contradiction. Suppose that there is a subsequence $k_l$ and a constant $N_0\in\N$ where $N_{k_l}\leq N_0$ for all $l\in\N$. In this case
$$\sum_{n = 1}^{N _0}\mu_\phi([\i(n)])\underline{\lambda}_m(\phi_{k_l},\i(n)) \geq (1+\delta)\lambda_m.$$
for all $l\in\N$. On the other hand, by Lemma \ref{lma:approximation}(2) we have for any $n \in \N$ that
$$\limsup_{k\to\infty}|\underline{\lambda}_m(\phi_k,\i(n))-\underline{\lambda}_m(\phi,\i(n))| \leq C_0 < \delta \lambda_m$$
since $m \geq M(\delta)$ and we fixed $M(\delta)$ such that $\delta \lambda_m > C_0$ for all $m \geq M(\delta)$ (recall property \ref{M1} again). Therefore as
$$\sum_{n = 1}^{N_0} \mu_\phi([\i(n)])\underline{\lambda}_m(\phi,\i(n)) < \lambda_m,$$
we have
$$\limsup_{l\to\infty}\sum_{n = 1}^{N_0} \mu_\phi([\i(n)])\underline{\lambda}_m(\phi_{k_l},\i(n)) \leq \delta \lambda_m + \sum_{n = 1}^{N_0} \mu_\phi([\i(n)])\underline{\lambda}_m(\phi,\i(n))  < (1+\delta) \lambda_m,$$
which is a contradiction. Thus we must have $N_k \to \infty$ as $k \to \infty$.

Since $N_k<\infty$ we can define
$$p_{\i(n),k} := \begin{cases}
0,&\text{if }n \geq N_k + 1;\\
\mu_\phi([\i(n)]),&\text{if }2 \leq n \leq N_k-1;\\
\frac{(1+\delta)\lambda_m - \sum\limits_{n = 1}^{N_k-1} \mu_\phi([\i(n)]) \underline{\lambda}_m(\phi_k,\i(n))}{\underline{\lambda}_m(\phi_k,\i(N_k))},&\text{if }n = N_k;\\
1-\sum\limits_{n=2}^{\infty}p_{\i(n),k},&\text{if } n = 1.
\end{cases}$$

Let us now define the numbers $r_i(k)$ such that they satisfy properties (i), (ii) and (iii), and then let us also check that they converge to $0$ for increasing $k$.
\begin{itemize}
\item[(i)] Define
$$r_1(k) := \big(p_{\i(1),k} - \mu_\phi([\i(1)])\big)\underline{\lambda}_m(\phi_k,\i(1)).$$
Then by the definition of the weights $p_{\i(n),k}$ we have
\begin{align*} \sum_{n=1}^{\infty} p_{\i(n),k}\underline{\lambda}_m(\phi_k,\i(n)) & = \sum_{n=1}^{N_k-1} \mu_\phi([\i(n)])\underline{\lambda}_m(\phi_k,\i(n)) \\
& \quad + \big(p_{\i(1),k} - \mu_\phi([\i(1)])\big)\underline{\lambda}_m(\phi_k,\i(1)) \\
& \quad +  p_{\i(N_k),k} \underline{\lambda}_m(\phi_k,\i(N_k)) \\
& = (1+\delta)\lambda_m + r_1(k).
\end{align*}
\item[(ii)] Define
\begin{align*}r_2(k) := & - p_{\i(1),k}\log p_{\i(1),k} + \mu_\phi([\i(1)])\log \mu_\phi([\i(1)]) \\
& - p_{\i(N_k),k}\log p_{\i(N_k),k} + \sum_{n = N_k}^\infty \mu_\phi([\i(n)])\log \mu_\phi([\i(n)]).\end{align*}
Then again
\begin{align*} - \sum_{n=1}^{\infty} p_{\i(n),k}\log p_{\i(n),k}\ & = -\sum_{n=1}^{\infty} \mu_\phi([\i(n)])\log \mu_\phi([\i(n)]) + r_2(k).
\end{align*}
\item[(iii)] Define
\begin{align*}r_3(k) & := \big(p_{\i(1),k} - \mu_\phi([\i(1)])\big)\overline{\lambda}_m(\phi,\i(1)) + p_{\i(N_k),k} \overline{\lambda}_m(\phi,\i(N_k)) \\
& \quad\,\, - \sum_{n = N_k}^\infty \mu_\phi(I_{\i(n)}) \overline{\lambda}_m(\phi,\i(n)).\end{align*}
Then recalling that $\lambda_m$ is defined by
$$\lambda_m = \sum_{n = 1}^\infty \mu_\phi(I_{\i(n)}) \overline{\lambda}_m(\phi,\i(n)),$$
we can use the definition of the weights $p_{\i(n),k}$ to obtain the following
\begin{align*} \sum_{n=1}^{\infty} p_{\i(n),k}\overline{\lambda}_m(\phi,\i(n)) & = \sum_{n = 1}^\infty \mu_\phi(I_{\i(n)}) \overline{\lambda}_m(\phi,\i(n)) \\
& \quad + \big(p_{\i(1),k} - \mu_\phi([\i(1)])\big)\overline{\lambda}_m(\phi,\i(1))  \\
& \quad + p_{\i(N_k),k} \overline{\lambda}_m(\phi,\i(N_k)) - \sum_{n = N_k}^\infty \mu_\phi(I_{\i(n)}) \overline{\lambda}_m(\phi,\i(n)). \\
& = \lambda_m + r_3(k).
\end{align*}
\end{itemize}
By the definition of $N_k$, observe that
$$0<p_{\i(N_k),k} = \frac{(1+\delta)\lambda_m - \sum\limits_{n = 1}^{N_k-1} \mu_\phi([\i(n)]) \underline{\lambda}_m(\phi_k,\i(n))}{\underline{\lambda}_m(\phi_k,\i(N_k))} \leq \mu_\phi([\i(N_k)]).$$
Moreover,
$$p_{\i(1),k} = \mu_\phi([{\i(1)}]) + t_k - p_{\i(N_k),k},$$
where we have defined $t_k$ to be the tail of the distribution $\mu_\phi$, that is
$$t_k := 1-\sum_{n = 1}^{N_k-1} \mu_\phi([\i(n)]).$$
Since $N_k \to \infty$ and so $\mu_\phi([\i(N_k)]) \to 0$, we have that as $p_{\i(N_k),k}\leq \mu_\phi([\i(N_k)])$, both
$$p_{\i(N_k),k} \to 0 \quad \text{and} \quad t_k \to 0$$
as $k \to \infty$. Furthermore, by Lemma \ref{lma:approximation}(2) there exists $C_0 > 0$ such that for each $n \in \N$ we have
$$\limsup_{k \to \infty}|\underline{\lambda}_m(\phi_k,\i(n)) - \underline{\lambda}_m(\phi,\i(n))| \leq C_0$$
and $\underline{\lambda}_m(\phi,\i(n)) < \infty$ for all $n$. Therefore
$$r_1(k),r_2(k),r_3(k) \to 0,\ \text{as } \quad k \to \infty,$$
and so the lemma is proved.

\end{proof}

Recall that $r_1(k),r_2(k),r_3(k) \to 0$ and they implicitly depend on $m$, but the convergence to zero will happen for any fixed $m \in \N$. Fix $m \in \N$ and choose $K(m) \in \N$ such that for any $k \geq K(m)$ we have
$$ |r_1(k)|,|r_2(k)|,|r_3(k)| \leq \min\{C, \delta h(\mu_{\phi},\sigma)\},$$
and
$$|r_1(k)-(1+\delta)r_3(k) | \leq \delta,$$
where $C$ was defined in \eqref{Cdef}.

We are now in a position to prove Lemma \ref{lma:step1}.
\begin{proof}[Proof of Lemma \ref{lma:step1}]
Fix $\delta \in (0, 1/3)$, $m \geq M(\delta)$ and $k\geq k(m)$. We first suppose that we are in the first case of Lemma \ref{lma:step1bernoulli}. In this case we can fix $\mu_k^m :=\mu_{\phi}$ which will be $\sigma^k$-ergodic since it is Gibbs for $\sigma$. We have that
$$-\int S_m\phi_k\, d\mu_{\phi}>-\int S_m\phi\, d\mu_{\phi}$$
and
$$\frac{h(\mu_{\phi},\sigma^k)}{-\int S_m\phi_{k}\d\mu_{\phi}}\geq\frac{1}{1+2\delta}\geq\frac{1-3\delta}{1+3\delta}.$$

If we are in the second case of Lemma \ref{lma:step1bernoulli}, we let $\mu_k^m$ be the $\sigma^m$ Bernoulli measure defined by the weights $(p_{\i,k})_{\i \in \N^m}$ from Lemma \ref{lma:step1bernoulli}. By the properties (i) and (iii) in Lemma \ref{lma:step1bernoulli} and the assumption \ref{M4} on $M(\delta)$, we have that
\begin{eqnarray*}
-\int S_m\phi_k\d\mu_k^m&\geq&\sum\limits_{\i \in \N^m} p_{\i,k}\underline{\lambda}_m(\phi_k,\i(n))\\
&=&(1+\delta)\lambda_m+r_1(k)\\
&\geq&(1+\delta)\left(\sum\limits_{\i \in \N^m} p_{\i,k}\overline{\lambda}_m(\phi,\i(n))\right) +r_1(k)-(1+\delta)r_3(k) \\
&\geq&(1+\delta)\left(\sum\limits_{\i \in \N^m} p_{\i,k}\overline{\lambda}_m(\phi,\i(n))\right)-\delta\\
&\geq&-(1+\delta)\int S_m\phi\d\mu_k^m-\delta\\
&>&-\int S_m\phi\d\mu_k^m.
\end{eqnarray*}
For the dimension we need an estimate in the opposite direction. By  property \ref{M3} of the choice of $M(\delta)$ we have
\begin{eqnarray*}
-\int S_m\phi_k\d\mu_k^m&\leq&\sum\limits_{\i \in \N^m} p_{\i,k}\underline{\lambda}_m(\phi_k,\i(n))+C\\
&=&(1+\delta)\lambda_m+r_1(k)+C\\
&\leq&(1+3\delta)(mh(\mu_{\phi},\sigma)).
\end{eqnarray*}
We also need an estimate on the entropy. Using property \ref{M3} of the choice of $M(\delta)$ once again, we have that
\begin{eqnarray*}
h(\mu_k^m,\sigma^m)&=&-\sum\limits_{\i \in \N^m}p_{\i,k}\log p_{\i,k}\\
&=&-\sum\limits_{\i \in \N^m}\mu_{\phi}([\i])\log\mu_{\phi}([\i])-r_3(k)\\
&\geq& mh(\mu_{\phi},\sigma)(1-2\delta).
\end{eqnarray*}
Putting these two estimates together, we obtain
$$\frac{h(\mu_k^m,\sigma^m)}{-\int S_m\phi_k\d\mu_k^m}\geq\frac{1-3\delta}{1+3\delta}.$$
Thus the proof is complete.

 \end{proof}

Now let us proceed with the proof of Theorem \ref{mainthm}. Let $\delta > 0$ and fix $m \geq M(\delta)$ (recall the choice of $M(\delta)$ from Lemma \ref{lma:step1}) and write
$$\psi_k := \tfrac{1}{m} S_m \phi_k \quad \text{and}\quad \psi := \tfrac{1}{m} S_m \phi$$
and define the auxiliary $\sigma$ invariant measure
$$\eta := \frac{1}{m}\sum_{i=0}^{m-1} \sigma^i \mu_k^m,$$
where $\mu_k^m$ is the $\sigma^m$ Bernoulli measure determined in Lemma \ref{lma:step1}. The measure $\eta$ satisfies the following properties:
$$\int\phi_k\d\eta=\int \psi_k \d\mu_k^m, \quad \int\phi\d\eta = \int \psi \d\mu_k^m \quad \text{and} \quad h(\eta,\sigma)=\frac{1}{m}h(\mu_k^m,\sigma)$$
and the dimension
$$s_k := \Hd\pi_k\eta =\frac{h(\eta,\sigma)}{\int \phi_k\d\eta} = \Hd\pi_{k}\mu_{k}^m.$$
Lemma \ref{lma:step1} will allow us to deduce the following lower bound on the pressure function
$$q\to P(q(\phi_k-\phi)-t\phi_k)$$
with a suitable choice of $t$.

\begin{lem}
\label{lem:pressure}
 If  $0 < t < s_k$, then
\[
\inf_{q\in \R} P(q(\phi_k-\phi)+t\phi_k) >0.
\]
\end{lem}

\begin{proof}
By Lemma \ref{lma:step1}, we have
$$-\int \phi_k \d\eta>-\int \phi \d\eta.$$
Thus we have that for all $q\leq 0$ the following property
$$\int [q(\phi_k-\phi)+t\phi_k]\d\eta+h(\eta,\sigma)>t\int\phi_k\d\eta+h(\eta,\sigma)>0.$$
On the other hand, if $q>0$ we first suppose that the potential $\phi_k$ has an equilibrium state $\nu_k$. In this case as $t < s_k \leq 1$ and $\int \phi_k - \phi \, d\nu_k > 0$ we have
$$\int  q(\phi_k-\phi)+t\phi_k\d\nu_{k}+h(\eta,\sigma)>t\int\phi_k\d\nu_{k}+h(\nu_{k},\sigma)>0.$$
Thus by the variational principle,
$$P(q(\phi_k-\phi)+t\phi_k)>\max\Big\{t\int\phi_k\d\eta+h(\eta,\sigma),t\int\phi_k\d\nu_{k}+h(\nu_{k},\sigma)\Big\}>0.$$
If $\phi_k$ does not have an equilibrium state then we must have that
$$\inf\{s:P(s\phi_k)=\infty\}=1$$
and by assumption
$$0\leq\inf\{s:P(s\phi_k)=\infty\}<1.$$
Therefore,   if we let $1>s>\max\{\inf\{s:P(s\phi)=\infty\},t\}$ and apply the first part of Lemma \ref{cusp} to $\phi_k$ and the second part to $\phi$,  we can find a compactly supported $\sigma$ invariant ergodic measure $\mu$ such that
$$h(\mu,\sigma)+s\int\phi_k\d\mu\geq 0$$
and
$$h(\mu,\sigma)+s\int\phi\d\mu\leq 0.$$
Therefore $\int\phi_k\d\mu\leq\int\phi\d\mu$ and so for all $q\leq 0$
$$\int  q(\phi_k-\phi)+t\phi_k\d\mu+h(\mu,\sigma)>0.$$
\end{proof}

We can now use the approximation property of pressure to allow us to find suitable measures which are compactly supported. Recall that the finite approximation property was given in Lemma \ref{finiteapprox}, and it states that $P(\phi) = \lim_{N\to \infty} P_N(\phi)$, where $P_N(\phi)$ is the pressure of $\phi$ restricted to the finite shift $\{1,2,\dots,N\}^\N$.

\begin{lem}
\label{lem:pressurefinite}
If  $0 < t < s_k$, then then there exists $N \in \N$ with
$$\inf\{P_N(q(\phi_k-\phi)-t\phi_k):q\in\R\}>0$$
and
$$\lim_{q\to\infty} P_N(q(\phi_k-\phi)-t\phi_k)=\lim_{q\to-\infty} P_N(q(\phi_k-\phi)-t\phi_k)=\infty.$$
\end{lem}
\begin{proof}
First of all by taking $\nu_k$ as in the proof of previous Lemma \ref{lem:pressure} we have
$$\int (\phi_k-\phi) \d\nu_k < 0 \quad \text{and} \quad \int (\phi_k-\phi) \d\eta > 0.$$
Let us use these measures $\eta$ and $\nu_k$ to construct measures $\tau_1$ and $\tau_2$ satisfying similar properties but supported on a compact set $\Sigma_N$ for a large enough $N$ as follows. By Birkhoff's ergodic theorem there exist words $\i,\j\in\Sigma$ and indices $n_1,n_2\in\N$ such that
$$\sigma^{n_1}\i=\i, \quad \sigma^{n_2}\j=\j, \quad S_{n_1}(\phi_k-\phi)(\i)>0, \quad \text{and}\quad  S_{n_2}(\phi_k-\phi)(\j)<0.$$
Thus if we let $\tau_1$ and $\tau_2$ be the measures supported on these $n_1$ and $n_2$ periodic orbits of $\i$ and $\j$ respectively, then there exists an index $M\in\N$ such that both $\tau_1,\tau_2$ are invariant measures on $\Sigma_N$ for all $N\geq M$ and we will have that
$$\int (\phi_k-\phi)\d\tau_1<0 \quad \text{and} \quad \int (\phi_k-\phi)\d\tau_2>0.$$
Thus if $N\geq M$ and we put
$$q_1 := \frac{t\int \phi_k\d\tau_1}{\int (\phi_k-\phi)\d\tau_1} \quad \text{and} \quad q_2 := \frac{t\int \phi_k\d\tau_2}{\int (\phi_k-\phi)\d\tau_2},$$
then by the variational principle there exists $C>0$ such that $P_N(q(\phi_k-\phi)-t\phi_k)>C$ for all $q\notin [2q_1,2q_2]$ and
$$\lim_{q\to\infty} P_N(q(\phi_k-\phi)-t\phi_k)=\lim_{q\to-\infty} P_N(q(\phi_k-\phi)-t\phi_k)=\infty.$$
On the other hand, by the finite approximation property (Lemma \ref{finiteapprox}) and Lemma \ref{lem:pressure} we have that
$$\lim_{n\to \infty} P_N(q(\phi_k-\phi)-t\phi_k)=P(q(\phi_k-\phi)-t\phi_k)\geq\inf_{q \in \R}P(q(\phi_k-\phi)-t\phi_k)>0$$
for all $q\in [2q_1,2q_2]$. Now if for each $n\in\N$ we define the set
$$Q_N := \{q\in [2q_1,2q_2]:P_N(q(\phi_k-\phi)-t\phi_k)\leq 0\},$$
then $Q_{N+1}\subset Q_N$ for for all $n\in\N$. However, if we can find $q\in \bigcap_{N=1}^{\infty}Q_N$, then $P(q(\phi_k-\phi)-t\phi_k)\leq 0$, which is a contradiction. Thus $\bigcap_{N=1}^{\infty}Q_N=\emptyset$ and since each $Q_N$ is compact there must exists $N\geq M$ such that $Q_N=\emptyset$. For this value of $N\in\N$ we must have that
$$\inf\{P_N(q(\phi_k-\phi)-t\phi_k):q\in\R\}>0$$
as claimed.
\end{proof}

Now for the $N \in \N$ constructed in Lemma \ref{lem:pressurefinite}, we can formulate a key lemma:

\begin{lem}\label{step2}
If  $0 < t< s_k$, then there exists $N\in\N$ such that
\begin{enumerate}
\item $\phi_k-\phi$ is not a coboundary on $\Sigma_N$.
\item there exists a Gibbs measure $\mu$ on $\Sigma_N$ such that
$$\int \phi_k \d\mu=\int \phi\d\mu \quad \text{and} \quad \frac{h(\mu,\sigma)}{-\int \phi_k\d\mu}>t.$$
\end{enumerate}
\end{lem}
\begin{proof}
By Lemma \ref{lem:pressurefinite} we know that there exists $N\in\N$ such that
$$\inf\{P_N(q(\phi_k-\phi)-t\phi_k):q\in\R\}>0$$
and
$$\lim_{q\to\infty} P_N(q(\phi_k-\phi)-t\phi_k)=\lim_{q\to-\infty} P_N(q(\phi_k-\phi)-t\phi_k)=\infty.$$
The restrictions of $\phi_k$ and $\phi$ to $\Sigma_N$ are H\"{o}lder continuous and so the function $Z_N:\R\to\R$ defined by
$$Z_N(q):=P_N(q(\phi_k-\phi)-t\phi_k)$$
is analytic with
$$Z_N'(q)=\int (\phi_k-\phi)\d\mu_q$$
by Lemma \ref{lem:thermoderivative}, where $\mu_q$ is the Gibbs measure on $\Sigma_N$ for $q(\phi_k-\phi)+t\phi_k$.

Since $\lim_{q\to\infty} Z_N(q)=\lim_{q\to-\infty} Z_N(q)=\infty$ we know by the definition of pressure that $\phi_k-\phi$ cannot be a coboundary on $\Sigma_N$. Therefore, as $\inf\{Z_N(q):q\in\R\}>0$, there must exist $q_1 \in \R$ such that $Z_N'(q_1)=0$. Thus the Gibbs measure $\mu := \mu_{q_1}$ on $\Sigma_N$ satisfies
$$\int (\phi_k-\phi)\d\mu=0$$
and by the variational principle (since $Z_N(q_1) > 0$) we have
$$h(\mu,\sigma)+t\int \phi_k\d\mu>0.$$
Therefore, we have by the negativity of $\phi_k$ that
$$\frac{h(\mu,\sigma)}{-\int \phi_k\d\mu}>t$$
as claimed.
\end{proof}

The key to the proof of the main theorem will be to combine the above result with the following simple application of the law of the iterated logarithm for function differences $f-g$, which are not coboundaries.
\begin{lem}\label{lil}
Let $f,g:\Sigma_N\to\R$ be H\"{o}lder continuous potentials such that $f-g$ is not a coboundary and let $\mu$ be a Gibbs measure on $\Sigma_N$ where $\int f\d\mu=\int g\d\mu$. We then have that
$$\liminf_{n\to\infty}e^{S_n(f-g)(x)}=0\quad \text{ and } \quad \limsup_{n\to\infty}e^{S_n(f-g)(x)}=\infty$$
for $\mu$ almost all $x\in\Sigma_N$.
\end{lem}
\begin{proof}
Since $f-g$ is not cohomologous to a constant we can apply the law of the iterated logarithm, Lemma \ref{lem:lawofiterated}, to the functions $f-g$ and $g-f$ to conclude that for some positive constants $c_1,c_2 > 0$ the following asymptotic bounds hold:
$$\liminf_{n \to \infty} \frac{S_n(f-g)(x)}{\sqrt{n \log\log n}} < -c_2  \quad \text{ and } \quad \limsup_{n \to \infty} \frac{S_n(f-g)(x)}{\sqrt{n \log\log n}} > c_1$$
at $\mu$ almost every $x\in\Sigma_N$. In particular at these $x$ also
$$\liminf_{n\to\infty}e^{S_n(f-g)(x)}=0 \quad \text{ and } \quad \limsup_{n\to\infty}e^{S_n(f-g)(x)}=\infty.$$
\end{proof}

Let us now complete the proof of the main theorem.

\begin{proof}[Proof of Theorem \ref{mainthm}] For any $0 < \delta < 1/3$ and $m \geq M(\delta)$ by Lemma \ref{lma:step1}, we can find $K = K(m)\in\N$ such that for all $k\in\N$ with $k\geq K$ there exists a $\sigma^m$-invariant ergodic measure $\mu_k^m$ on $\Sigma$ such that
$$\int (\psi_k-\psi)\d\mu_k^m>0 \quad \text{and} \quad \frac{\frac{1}{m}h(\mu^m_k)}{-\int\psi_k\d\mu_k^m} = \dim \pi_{k}\mu_k^m > \frac{1-3\delta}{1+3\delta}.$$
Thus by Lemma \ref{step2} applied to $t = (1-2\delta)/(1+2\delta)$ and for  the $N \in \N$ given by that result, $\phi_k-\phi$ is not a coboundary on $\Sigma_N$ and we can find a Gibbs measure $\mu$ supported on a compact set of $\Sigma$ (i.e. $\Sigma_N$ embedded into $\Sigma$) such that
$$\int \psi_k \d\mu = \int \psi\d\mu \quad \text{and} \quad \dim \mu\circ\pi_{k}  > \frac{1-3\delta}{1+3\delta}.$$
Therefore, by Lemma \ref{lil}, we may also assume that at $\mu$ almost all $x\in\Sigma$ we have
$$\liminf_{n\to\infty}e^{S_n(\psi_k-\psi)(x)}=0\text{ and }\limsup_{n\to\infty}e^{S_n(\psi_k-\psi)(x)}=\infty.$$
Fix one such $x \in \Sigma$. Recall that the projections $\pi_{k},\pi : \Sigma \to [0,1]$ map cylinder sets from  $\Sigma$ onto $T_k$ and $T$ construction intervals respectively and the conjugacy $\theta_k$ between $T_k$ and $T$ satisfies
$$\theta_k(\pi_{k}(x)) = \pi(x).$$
Now for each $n \in \N$, let us define a word $y = y(n) \in \N^{n+1}$ by
$$y := \begin{cases}x|_n3, &\text{ if }x_{n+1} = 1;\\
x|_n4, &\text{ if }x_{n+1} = 2;\\
x|_n 1, &\text{ if } x_{n+1} \geq 3.
\end{cases}$$
Then $\pi_k(y) \in I_{x_1,\dots,x_n}^{(T_k)}$ and so $\theta_{k}(\pi_k(y)) \in I_{x_1,\dots,x_n}^{(T)}$, where we emphasise the interval map $T_k$ or $T$ used. Therefore, for all $n \in \N$ the distances
$$|\pi_{k}(x)-\pi_k(y)|\leq |I_{x_1,\dots,x_n}^{(T_k)}| = e^{S_n\psi_k(x)}$$
and
$$|\theta_k(\pi_{k}(x))-\theta_k(\pi_k(y))|\leq |I_{x_1,\dots,x_n}^{(T)}| = e^{S_n\psi(x)}.$$
Moreover, we have the lower bound
$$|\pi_{k}(x)-\pi_k(y)| \geq \begin{cases} |I_{x_1,\dots,x_n2}^{(T_k)}|, &\text{ if }x_{n+1} = 1;\\
 |I_{x_1,\dots,x_n3}^{(T_k)}|, &\text{ if }x_{n+1} = 2;\\
|I_{x_1,\dots,x_n2}^{(T_k)}|, &\text{ if } x_{n+1} \geq 3.
\end{cases}$$
so in all cases there is $c_k = c_k(x) > 0$ independent of $n$ satisfying
$$|\pi_{k}(x)-\pi_k(y)| \geq c_k e^{S_n\psi_k(x)}.$$
Similarly, for a suitable $c = c(x) > 0$ independent of $n$ the images satisfy
$$|\theta_k(\pi_{k}(x))-\theta_k(\pi_k(y))| \geq c e^{S_n\psi(x)}$$
Thus as the numbers $c_k$ and $c$ are independent of $n$ we obtain by our choice of $x$ that
$$\liminf_{n\to\infty} \frac{|\theta_k(\pi_{k}(x))-\theta_k(\pi_k(y))|}{|\pi_{k}(x)-\pi_k(y)|}\leq \liminf_{n\to\infty} c_k^{-1}e^{S_n(\psi_k-\psi)(x)}=0$$
and
$$\limsup_{n\to\infty} \frac{|\theta_k(\pi_{k}(x))-\theta_k(\pi_k(y))|}{|\pi_{k}(x)-\pi_k(y)|}\geq \limsup_{n\to\infty} ce^{S_n(\psi_k-\psi)(x)}=\infty.$$
Thus the derivative of $\theta_k$ at $\pi_{k}(x)$ cannot exist. Since $x$ was $\mu$ typical, this means that $\mu\circ \pi_{k}$ gives full mass to the set of $y$ where $\theta_k'(y)$ does not exist. Therefore, for all $k \geq K$ we have
$$\Hd \{y\in [0,1]:\theta_k'(y)\text{ does not exist}\} \geq \dim\pi_{k} \mu >  \frac{1-3\delta}{1+3\delta}.$$
The proof of Theorem \ref{mainthm} is therefore complete,  since $1/3>\delta > 0$ was chosen arbitrarily.
\end{proof}

\section{Manneville-Pomeau maps}\label{sec:manpom}

Let us now prove Corollary \ref{cor:mp} to Theorem \ref{mainthm}. Fix $\alpha,\beta > 0$ with $\alpha \neq \beta$ and let $\widehat{M}_\alpha$ and $\widehat{M}_\beta$ be the jump transformations of $M_\alpha$ and $M_\beta$. That is, if $r_\alpha(x) \in \N$ is the first hitting time to the interval between $[b_\alpha,1]$, where $b_\alpha$ is the solution to the equation $x+x^{1+\alpha} = 1$ on $(0,1)$, then
$$\widehat{M}_\alpha(x) := M_\alpha^{r_\alpha(x)}(x)$$
and similarly for $\widehat{M}_\beta$. Now the topological conjugacy $\theta_{\alpha,\beta}$ between $M_\alpha$ and $M_\beta$ agrees with the topological conjugacy between $\widehat{M}_\alpha$ and $\widehat{M}_\beta$. Therefore, in order to prove Corollary \ref{cor:mp}, we need to establish the assumptions on Theorem \ref{mainthm} when $\beta \to \alpha$.

(a) Pointwise convergence of the inverse branches of the induced maps can be established since when $\beta \to \alpha$, we have that $M_\beta(x) \to M_\alpha(x)$ and the hitting times $r_\beta(x) \to r_\alpha(x)$ for a fixed $x \in [0,1]$.

(b) Now for the tail behaviour, that is, condition (1) in Theorem \ref{mainthm}, we will cite Sarig \cite{sarigCMP} and in particular the proof of Proposition 1 there, where it is proved that if $f_i$ are the inverse branches of $\widehat{M}_\alpha$, then for any $0 < \alpha < \infty$ there exists $t(\alpha) > 0$ with
$$\sum_{i = 1}^\infty |f_i[0,1]|^{t(\alpha)} < \infty.$$

(c) Finally, the variations will be uniformly bounded. Fix any $\eps > 0$ such that $\alpha - \eps > 0$. For $\beta > 0$, write
$$\phi_\beta(\i) := -\log |\widehat{M}_\beta'(\pi_{\widehat{M}_\beta}(\i))|,$$
where we  recall that $\pi_{\widehat{M}_\beta}$ maps cylinders $[\i]$ onto intervals $I^{\widehat{M}_\beta}_\i$. Then to check the uniform bound (2) in Theorem \ref{mainthm} on variations, we will need to establish
$$\sup_{\beta \in I(\alpha)} \sum_{n = 1}^\infty \var_n(\phi_\beta) < \infty,$$
where $I(\alpha) := [\alpha-\eps,\alpha+\eps] \subset (0,\infty)$ as this yields the assumption (2) in Theorem \ref{mainthm} for all sequences $\widehat{M}_{\beta_k}$, where $\beta_k \to \alpha$ as $k \to \infty$. To do this, we just need to check that the mapping $\beta \mapsto \sum_{n = 1}^\infty \var_n(\phi_\beta)$ is bounded by a continuous function since the supremum is over a compact interval $I(\alpha)$. This follows from Nakaishi's work \cite[Lemmas 2.1 and 2.2]{nakaishi} where the following estimate can be established:
$$|\phi_\beta(\j) - \phi_\beta(\k)| \leq C(\beta) n^{-p(\beta)}$$
for $\i \in \N^n$ and $\j,\k \in [\i]$ and so $\var_n(\phi_\beta) \leq C(\beta)n^{-p(\beta)}$. Here the constants $C(\beta) > 0$ and $p(\beta) > 1$ depend continously on the parameter $\beta$. Hence $\sum_{n = 1}^\infty \var_n(\phi_\beta) \leq C(\beta) \zeta(p(\beta))$, where $\zeta$ is the Riemann zeta function. Thus the sum is bounded by a continuous function of $\beta$, which is what we wanted.

\section{H\"older exponents, dimension of $\mu \circ \theta_k$ and the entropy} \label{sec:holderentropy}

In this section we will prove Propositions \ref{prop:holder}, \ref{prop:hausdorff} and \ref{prop:lyapunov} by giving examples of countable Markov maps $T_k$ and $T$ satisfying the conditions of Theorem \ref{mainthm} but with, respectively,  the H\"older exponents, Hausdorff dimensions of the push-forward of the invariant measure for $T$  and Lyapunov exponents failing to converge. All of the examples we give below come from the class of $\alpha$-L\"uroth maps, which were introduced in \cite{KMS}, so let us briefly recall the definition. We start with a sequence of real numbers $0 < t_k \leq 1$ with the property that $\lim_{k\to\infty}t_k=0$ and let $\alpha:=\{A_n:=(t_{n+1}, t_n]:n\in\N\}$. We also denote the length of $A_n$ by $a_n:=a_n(\alpha)$. Then the map $\alpha$-L\"uroth map $L_\alpha$ is defined to be the countable Markov map with inverse branches that map the unit interval affinely onto each partition element $A_n$. Two particular examples we will use below come from the partitions $\alpha_L$, defined by $t_n:=1/n$, and $\alpha_D$, which is given by $t_n:=2^{-(n-1)}$.

\subsection{H\"older exponents}\label{sec:holder}

We start with the map $T:=L_{\alpha_D}$ as described above. Then we modify the partition $\alpha_D$ to obtain a sequence of $\alpha$-L\"uroth maps that converge pointwise to $T$, in the following way. Let $\alpha_k$ be the partition where $a_n(\alpha_k)= a_n(\alpha_D)$ for all $n\notin\{k, k+1\}$,  and we modify the point $t_{k+1}(\alpha_D)$ in order to obtain the lengths $a_k(\alpha_k) = 2^{-k^2}$ and $a_{k+1} = 2^{-k}+2^{-(k+1)}-2^{-k^2}$. Then the conjugacy map $\theta_k$ between $T_k$ and $T$ is exactly the map studied in \cite{KMS}, where in particular it was shown in \cite[Lemma 2.3]{KMS} that the H\"older exponent of $\theta_k$ is given by
\[
\kappa(\theta_k) = \inf\left\{\frac{\log a_n(\alpha_D)}{\log a_n(\alpha_k)}:n\in\N\right\}.
\]
Therefore, for our example, we see that the H\"older exponent of $\theta_k$ is given by $1/k$. This proves Proposition \ref{prop:holder}.

\subsection{Hausdorff dimension of $\mu \circ \theta_k$} \label{sec:hausdorff}  In this case we choose $T$ to be the $\alpha_L$-L\"uroth map, so $a_n (\alpha_L)= 1/(n(n+1))$ for all $n\in \N$.  Therefore we  have that the Lyapunov exponent and the entropy
$$\lambda(\mu, T)=h(\mu, T) = \sum_{i = 1}^\infty -a_i \log a_i < +\infty.$$
Now for each $k\in\N$ we make a modification to the partition $\alpha_L$ to obtain a sequence of partitions $\alpha_k$ as follows. Fix the first $k$ elements of the partition, and then for $i>k$ let the partition elements have size
\[
a_i(\alpha_k)= \frac1{(k+1)2^{i-k}}.
\]
Letting $T_k:=L_{\alpha_k}$, and the conjugacy between $T_k$ and $T$ again be denoted by $\theta_k$, the conditions of Theorem \ref{mainthm} are readily seen to hold as $-\log|T_k'|$ is a piecewise constant function and the tail $t_i$ decays exponentially. However, for each $k$ we have that
$$h(\mu\circ \theta_k, T_k) = h(\mu,T)  < +\infty,$$
but the maps $T_k$ are constructed such that
$$\lambda(\mu\circ \theta_k, T_k) = \sum_{i = 1}^\infty -a_i(\alpha_L) \log (a_i(\alpha_k)) = +\infty.$$
An application of Proposition \ref{prop:measdim} now finishes the proof of Proposition \ref{prop:hausdorff}.

\subsection{Entropy}\label{sec:entropyconvergence}

An example of maps where the Lyapunov exponents to fail to converge is made by adapting the tails of the partition $\alpha_D$ again, similarly to the trick for H\"older exponents in Section \ref{sec:holder}. So, let $T:=L_{\alpha_D}$, and recall that this means $a_i = 2^{-i}$ for all $i\in \N$. 
Thus for the entropy we have
$$h(\mu,T) = \sum_{i = 1}^\infty -2^{-i} \log 2^{-i} < \infty.$$
Now let us define a partition $\alpha_k$ by fixing the first $k-1$ elements to be equal to the first $k-1$ elements from the dyadic partition, letting
\[
A_{k, k}:=\left(\frac{t_k}{\log 2},t_k\right],
\]
and, for $i>k$, letting
\[
A_{k, i}:=\left(\frac{t_k}{\log (n+2-k)},\frac{t_k}{\log (n+1-k)}\right].
\]
Thus, for $i>k$, we have that
\[
a_i(\alpha_k)= t_k \frac{\log(i+2-k)-\log(n+1-k)}{\log (i+2-k)\log (i+1-k)},
\]
and these decay for any fixed fixed $k \in \N$ with the rate $O(1/\log i)$ as $i \to \infty$, which is far too slow to have finite entropy for $\mu_k$. This yields that the entropy $h(\mu_k,T_k)= \infty$ for all $k \in \N$.

\section{Computing the specific value of the Hausdorff dimension}\label{sec:geometric}

In this section, we first aim to prove Theorem \ref{mainthm2}. Before we begin, we must introduce some preliminaries and notation (for more details, we refer to \cite{JKPS} and the references therein).
To begin, suppose that we have two maps $S, T: [0, 1]\to [0,1]$ which have exactly two full branches, the first (thinking left to right from the origin) increasing and the second decreasing, and both branches are strictly contracting $C^{1+\varepsilon}$ diffeomorphisms.
We are interested, of course, in the topological conjugacy map $\theta$ between $S$ and $T$, and the set $\Theta:=\{x:\theta'(x)\neq0\}$. Below, to mirror \cite{JKPS} more closely, and to make it clearer where changing from increasing branches to the tent-like case we have here makes differences to the proof, we also use the sets $\mathcal{D}_\sim$ and $\mathcal{D}_\infty$, which are defined to be the set of points where the derivative of said conjugacy map does not exist or is infinite, respectively. Note that these are the only options for the derivative to be non-zero.

We define two H\"older continuous potentials $\phi, \psi:\{0, 1\}^\N\to \R_{<0}$ by setting
\[
\phi(x_1, x_2, x_3, \ldots):= \log|(S^{-1}_{x_1})'\pi_S(x_2, x_3, \ldots)|\ \text{ and } \ \psi(x_1, x_2, x_3, \ldots):= \log|(T^{-1}_{x_1})'\pi_T(x_2, x_3, \ldots)|.
\]
Also, to simplify the notation later, we define $\chi:=\psi - \phi$.
Then, where we recall that $P$ denotes the topological pressure, we can define a function $\beta:\R\to \R$ implicitly through the pressure equation
\[
P(s\phi+\beta(s)\psi) = 0.
\]
We let $\mu_s$ denote the equilibrium measure for the the potential $s\phi+\beta(s)\psi$, which always exists and is unique. Recall from the preliminaries given in Section 2 that this means $\mu_s$ achieves the supremum in the variational principle:
\[
P(s\phi+\beta(s)\psi)= \sup_{\mu\in M_{\sigma}}\left\{h(\mu,\sigma)+\int s\phi+\beta(s)\psi\ d\mu\right\} = h(\mu_{s})+\int s\phi+\beta(s)\psi \ d\mu_{s}.
\]
Further,
\[
\beta'(s):=\frac{-\int \phi\ d\mu_s}{\int \psi\ d\mu_s}<0.
\]
If we suppose that $\phi$ and $\psi$ are cohomologically independent, there also exists a unique $s_0$ such that $\beta'(s_0)=-1$. Let $\tilde{\beta}(s_0):= \beta(s_0)+s_0$.

\begin{proposition}\label{prop:salemthermo}
We have
\begin{eqnarray}\label{starp8}
0< \Hd(\mathcal{D}_\sim) = \Hd(\mathcal{D}_\infty) =\tilde{\beta}(s_0)<1.
\end{eqnarray}
\end{proposition}
After proving this proposition, we will show that the value $\tilde \beta (s_0)$ gives the sought-after value in Theorem \ref{mainthm2}, for the specific example contained there.

We will now give a sequence of lemmas which give the necessary geometric information about the derivative, and how the differential quotient can be transferred to a sort of ``symbolic derivative''. We also need some notation: We write $[x_1, \ldots, x_n]:=\{y=(y_1, y_2, \ldots):y_i=x_i\text{ for all }1\leq i \leq n\}$ for the symbolic cylinder sets and write $I(x_1, \ldots, x_n):=\pi([x_1, \ldots, x_n])$ for the projection of the cylinder set $[x_1, \ldots, x_n]$ to a subinterval of $[0, 1]$. We also recall the definition of the variations of $\phi$,
\[
\var_k(\phi):=\sup_{(i_1, \ldots, i_k)\in \{0, 1\}^\N}\sup_{x, y\in [i_1, \ldots, i_k]}|\phi(x) - \phi(y)|,
\]
and note that here, since we are in a compact metric space, $\var_0(\phi)$ is finite. Since the potentials $\phi$  and $\psi$ are H\"older continuous, the variations of both are exponentially decaying and thus summable.

\begin{lem}\label{neighbours}
There exists a constant $C>0$, independent of $n$, such that if $I(x_1,\ldots,x_n)\cap I(y_1,\ldots, y_n)\neq\emptyset$ and $\omega, \tau\in [x_1,\ldots,x_n]\cup [y_1,\ldots,y_n]$ then
$$|S_n\phi(\omega)-S_n\phi(\tau)|\leq C$$
and
$$|S_n\psi(\omega)-S_n\psi(\tau)|\leq C.$$
\end{lem}
\begin{proof}
Let $j:=\inf\{1\leq k\leq n: x_k\neq y_k\}$. Observe that for our maps with one increasing and one decreasing full branch, the projection of the cylinder sets from $\{0, 1\}^\N$ works as follows: If $(x_1, \ldots, x_k)\in \{0, 1\}^k$ is such that $\sum_{i=1}^k x_i$ is odd, then $I(x_1, \ldots, x_k)$ splits into $I(x_1, \ldots, x_k, 0)\cup I(x_1, \ldots, x_k, 1)$, written in order, left to right, whereas if the sum of the digits $x_i$ is even, the $(k+1)$-level cylinders project the other way around, namely, to $I(x_1, \ldots, x_k, 1)\cup I(x_1, \ldots, x_k, 0)$. This implies that in order for $I(x_1,\ldots,x_n)\cap I(y_1,\ldots, y_n)\neq\emptyset$, we must have that $(x_{j+1},\ldots,x_n)=(y_{j+1},\ldots,y_n)=(1,0,\ldots,0)$. That is, the words $(x_1, \ldots, x_n)$ and $(y_1, \ldots, y_n)$ can only be different at the $j$-th letter.
Thus we have that
$$|S_n\phi(\omega)-S_n\phi(\tau)|\leq \sum_{k=0}^{j-1} \mathrm{var}_k(\phi)+\sum_{k=1}^{n-j} \mathrm{var}_{k}(\phi)\leq 2\sum_{k=1}^{\infty}\mathrm{var}_k(\phi)$$
and
$$|S_n\psi(\omega)-S_n\psi(\tau)|\leq \sum_{k=0}^{j-1} \mathrm{var}_k(\psi)+\sum_{k=1}^{n-j} \mathrm{var}_{k}(\psi)\leq 2\sum_{k=1}^{\infty}\mathrm{var}_k(\psi).$$
(Here, the $k=0$ term occurs precisely for the one difference at the $j$-th letter.) Thus we can take $C=2\max\left\{\sum_{k=1}^{\infty}\mathrm{var}_k(\phi),\sum_{k=1}^{\infty}\mathrm{var}_k(\psi)\right\}.$
\end{proof}

Before stating the next lemma, we note that by $f\asymp g$, we mean there exists a constant $c>1$ such that $c^{-1}\cdot f\leq g\leq c\cdot f$.
\begin{lem}\label{BD}
For all $x$ we have that,
$$\mathrm{diam}([x_1,\ldots,x_n])\asymp e^{S_n\phi(x)}$$
and
$$\mathrm{diam}(\theta[x_1,\ldots,x_n])\asymp e^{S_n\psi(x)}.$$
\end{lem}
\begin{proof}
Let us denote the inverse branches of $S$ by $f_0$ and $f_1$. Then if we apply the mean value theorem to the map $f_{x_1}\circ \cdots\circ f_{x_n}$, it follows that there exists a point $z\in(0, 1)$ such that
\[
(f_{x_1}\circ \cdots\circ f_{x_n})'(z) = \text{diam}([x_1,\ldots,x_n]).
\]
 Then, if we set $y:=f_{x_1}\circ \cdots\circ f_{x_n}(z)$, we have that $y\in [x_1,\ldots,x_n]$ and
$$\text{diam}([x_1,\ldots,x_n])\asymp e^{S_n\phi(y)}.$$
The result then follows since
$$|S_n\phi(y)-S_n\phi(x)|\leq\sum_{k=1}^{\infty}\mathrm{var}_k(\phi).$$
The second part of the result follows by exactly the same method.
\end{proof}
We can now relate these results to the derivative at a point $x$.
\begin{lem}\label{derivative}
Fix $x\neq y\in [0,1]$ and let $n=\inf\{k:I(x_1,\ldots,x_k)\cap I(y_1,\ldots,y_k)=\emptyset\}.$ We have that
$$\frac{|\theta(x)-\theta(y)|}{|x-y|}\asymp e^{S_n\chi(x)}.$$
\end{lem}
\begin{proof}
First note that as the points $x$ and $y$ are different in the unit interval (not just having different symbolic codes), this $n$ always exists In other words, at level $n-1$ the points $x$ and $y$ are in neighbouring subintervals of $[0,1]$, and at level $n$ there is at least one interval between them.
By the previous two results it then follows that
$$|\theta(x)-\theta(y)|\asymp e^{S_n\psi(x)}$$
and
$$|x-y|\asymp e^{S_n\phi(x)}.$$
The result immediately follows.
\end{proof}

\begin{lem}\label{lem2.4}
\
\begin{itemize}
  \item[(a)]
  $\mathcal{D}_{\sim}\supseteq\{x:\limsup_{n\to\infty}e^{S_n\chi(x)}=\infty\text{ and }\liminf_{n\to\infty}e^{S_n\chi(x)}=0\}$,
  \item[(b)]
  $\mathcal{D}_{\sim}\subseteq\{x:\limsup_{n\to\infty}e^{S_n\chi(x)}>0\}$
\end{itemize}
\end{lem}
\begin{proof}
Let $x$ be such that $\limsup_{n\to\infty}e^{S_n\chi(x)}=\infty$ and $\liminf_{n\to\infty}e^{S_n\chi(x)}=0$. For each $n\in\N$ we can find $y(n)$ such that
$n=\inf\{k:[x_1,\ldots,x_k]\cap [y(n)_1,\ldots,y(n)_k]=\emptyset.$ By Lemma \ref{derivative} we have that
$$\frac{|\theta(x)-\theta(y(n))|}{|x-y(n)|}\asymp e^{S_n\chi(x)}.$$
Part (a) of the Lemma follows immediately.

To prove (b),  suppose that $\lim_{n\to \infty}e^{S_n\chi(x)}=0$. Let $y(n)$ be a sequence such that $\lim_{n\to\infty} y(n)=x$ where each $y(n)\neq x$. Let
$k(n)=\inf\{j:[x_1,\ldots,x_j]\cap [y(n)_1,\ldots,y(n)_j]=\emptyset\}.$ We have that $\lim_{n\to\infty}k(n)=\infty$ and
$$\frac{|\theta(x)-\theta(y(n))|}{|x-y(n)|}\asymp e^{S_{k(n)}\chi(x)}.$$
Thus since $\lim_{n\to\infty}e^{S_{k(n)}(x)}=0$ we have that $\theta'(x)=0$ and the proof is finished.
\end{proof}

\begin{rem}\label{remp10}
Notice that the proof of part (b) in the above lemma shows rather more, namely, that
\[
\Theta\subset \{x:\limsup_{n\to\infty}e^{S_n\chi(x)}>0\}.
\]
\end{rem}

Now, recalling the definition of $s_0$ given above it can be shown by adapting the methods in \cite{JKPS} (the start of Section 2.2) that
$$\dim_{\mathrm{H}}\left\{x:\limsup_{n\to\infty}e^{S_n\chi(x)}>0\right\}\leq\tilde{\beta}(s_0)$$
and by the arguments in Section 2.3 of \cite{JKPS} that
$$\dim_{\mathrm{H}}\left\{x:\liminf_{n\to\infty}e^{S_n\chi(x)}=0\text{ and }e^{S_n\chi(x)}=\infty\right\}\geq\tilde{\beta}(s_o)$$
and so
$$\dim_{\mathrm{H}}(\mathcal{D}_{\sim})=\tilde{\beta}(s_0).$$
The second part is essentially identical to the work in \cite{JKPS} so we only show here how to adapt the proof of
$$\dim_{\mathrm{H}}\left\{x:\limsup_{n\to\infty}e^{S_n\chi(x)}>0\right\}\leq\tilde{\beta}(s_0).$$

If $\limsup _{n\to\infty}S_n\chi(x)>0$,  setting $K_1:=\sum_{k=0}^{\infty}\var_k(\chi)$ and letting $x^{(n)}:=\overline{x_1,\ldots,x_n}$ be the periodic element in the cylinder set $[x_1,\ldots,x_n]$, we then have for infinitely many $n\in\N$ that \[
S_n\chi\left(x^{(n)}\right)>-K_1.
\] It follows that
$$e^{K_1}e^{S_n\psi(x^{(n)})}\geq e^{S_n\phi(x^{(n)})}$$
for infinitely many $n\in\N$.
So for all $N\in\N$ the union for $n\geq N$ of $n$th level cylinders $[x_1,\ldots,x_n]$ where $S_n\psi(x^{(n)})>S_n\phi(x^{(n)})-K_1$ forms a cover of the set $\{x:\limsup_{n\to\infty} e^{S_n\chi(x)}>0\}$. We will denote this cover by $C_n$.

We can also choose $\xi<1$ such that $\diam[x_1,\ldots,x_n]\leq\xi^n$ for any cylinder set $[x_1,\ldots,x_n]$. Thus

\begin{eqnarray*}
\mathcal{H}_{\xi^N}^{\beta(s_0)+s_0+\epsilon}(\mathcal{D}_{\sim})&\leq&\sum_{n\geq N}\sum_{C_n} |x_1,\ldots,x_n|^{\beta(s_0)+s_0+\epsilon}\ll \sum_{n\geq N}\xi^{n\epsilon}\sum_{C_n} e^{(\beta(s_0)+s_0)S_n\phi(x^{(n)})}\\
&\ll&\sum_{n\geq N}\xi^{n\epsilon}\sum_{C_n} e^{(\beta(s_0)+s_0)S_n\phi(x^{(n)})}\\
&\ll&\sum_{n\geq N}\xi^{n\epsilon}\sum_{C_n}e^{\beta(s_0)\phi(x^{(n)})+s_0\psi(x^{(n)})}.
\end{eqnarray*}

Now since $P(\beta(s_0)\phi+s_0\psi)=0$ we can choose $N$ sufficiently large such that for any $n\geq N$,
$$\frac{1}{n}\log \sum_{C_n}e^{\beta(s_0)\phi(x^{(n)})+s_0\psi(x^{(n)})}\leq \frac{1}{n}\log \sum_{(x_1,\ldots,x_n)\in\{0,1\}^n}  e^{\beta(s_0)\phi(x^{(n)})+s_0\psi(x^{(n)})}\leq -\epsilon\log((\xi+1)/2).$$

Thus for $N$ sufficiently large if $n\geq N$ then
$$\sum_{C_n}e^{\beta(s_0)\phi(x^{(n)})+s_0\psi(x^{(n)})}\leq \left(\frac{\xi+1}{2}\right)^{-n\epsilon}$$ and so
$$\sum_{n\geq N}\sum_{C_n} |x_1,\ldots,x_n|^{\beta(s_0)+s_0+\epsilon}\ll 1.$$
Thus for all $N\in\N$, $\mathcal{H}_{\xi^N}^{\beta(s_0)+s_0+\epsilon}(\mathcal{D}_{\sim})\ll 1$ and so, since $\epsilon>0$ was arbitrary,  $\Hd \mathcal{D}_{\sim}\leq \beta(s_0)+s_0$.

To finish the proof of (\ref{starp8}) in light of Remark \ref{remp10} it suffices to show that
\[
\Hd(\mathcal{D}_\infty)\geq \tilde{\beta}(s_0).
\]
However, this again can be done precisely as in \cite{JKPS}, so we omit the details.  Thus the proof of Proposition \ref{prop:salemthermo} is complete.

Let us now show how to prove Theorem \ref{mainthm2} using Proposition \ref{prop:salemthermo}. First, fix $0 < \tau,\tau' < 1$ and denote by $\theta = \theta_{\tau,\tau'}$ the conjugacy map between $T_\tau$ and $T_{\tau'}$.  In order to apply Proposition \ref{prop:salemthermo}, we observe that the conjugacy $\theta$ coincides with the conjugacy map between the Farey maps $F_{\tau}$ and $F_{\tau'}$, where
$F_{\tau}$ is defined for $x\in [0,1]$ by
$$F_{\tau}(x)  := \begin{cases}
x/\tau, & x < \tau\\
(x-1)/(1-\tau), & x \geq 1-\tau.
\end{cases}
$$
For more details on these maps, we refer to \cite{KMS}. That the conjugacies coincide is a direct consequence of the fact that $T_\tau$ is the jump transformation of $F_\tau$.

\begin{proof}[Proof of Theorem \ref{mainthm2}]
We apply Proposition \ref{prop:salemthermo} with $S:=F_{\tau}$ and $T := F_{\tau'}$ the tent map.
In this case the potentials $\phi := \phi_\tau$ and $\psi:= \phi_{\tau'}$ are given by
$$\phi_\tau(x_1, x_2, x_3, \ldots) =\left\{
                                 \begin{array}{ll}
                                   \log\tau, & \hbox{if $x_1=0$;} \\
                                   \log(1-\tau), & \hbox{if $x_1 = 1$.}
                                 \end{array}
                               \right.$$
and similarly for $\phi_{\tau'}$. Also, since these maps are linear, we have that the measure $\mu_{s_0}$ is given by the $(p_{\tau,\tau'}, 1-p_{\tau,\tau'})$-Bernoulli measure $\mu_{\tau,\tau'}$ such that
\[
\beta'(s_0) = \frac{\int \phi_\tau\, d\mu_{\tau,\tau'}}{\int \phi_{\tau'}\, d\mu_{\tau,\tau'}}=1.
\]
Let us find the precise value for this $p = p_{\tau,\tau'}$ as follows. The fraction is
\[
\frac{\int \phi_{\tau}\, d\mu_{\tau,\tau'}}{\int \phi_{\tau'}\, d\mu_{\tau,\tau'}} = \frac{p\log \tau+(1-p)\log(1-\tau)}{p\log \tau'+(1-p)\log(1-\tau')}
\]
and so
\[
p = \frac{\log (1-\tau') - \log(1-\tau)}{\log \tau - \log(1-\tau) - \log \tau' + \log (1-\tau')}
\]
Furthermore, as $\mu_{s_0}= \mu_\tau$ is an equilibrium measure, we have that
\[
\Hd \{x:\theta_{\tau,\tau'}'(x)\neq 0\} = \tilde{\beta}(s_0) = \frac{h(\mu_{s_0})}{-\int \phi\ d\mu_{s_0}}= \frac{p \log p+(1-p)\log(1-p)}{p\log \tau+(1-p)\log(1-\tau)}
\]
as claimed.
\end{proof}

\begin{rem}
The proof of Proposition \ref{prop:salemthermo} is similar to how in \cite{JKPS} the \textit{Salem family} $S_\tau$ is analysed, where the Salem maps are interval maps with two increasing branches with slopes $1/\tau$ and $1/(1-\tau)$.
\end{rem}



\end{document}